\documentclass[11pt]{article}
\usepackage{amsthm}
\usepackage{amsmath}
\usepackage{amssymb}

\newcommand{\fpd}[2]{{\displaystyle\frac{\partial #1}{\partial #2}}}

\newcommand{\vf}[1]{{\displaystyle\frac{\partial}{\partial #1}}}
\newcommand{\onehalf}{{\textstyle\frac12}}
\newcommand{\dte}[1]{\eval{\frac{d}{dt} \, #1}{t=0}}
\newcommand{\eval}[2]{\left. #1 \right|_{#2}}

\newcommand{\TMO}{T^\circ\!M}
\newcommand{\TxMO}{T_x^\circ\!M}
\newcommand{\TyMO}{T_{y}^\circ\!M}
\newcommand{\Sec}{\mathop{\mathrm{Sec}}}

\newcommand{\conn}[3]{\setbox1=\hbox{$\scriptstyle{#2}{#3}$}
\setbox2=\hbox to\wd1{$\hfil\scriptstyle{#1}\hfil$}
\Gamma^{\!\box2}_{\!\box1}}

\newcommand{\C}{\mathfrak{C}}
\newcommand{\Cur}{\mathcal{C}}
\newcommand{\diff}{\mathop{\mathrm{diff}}}
\newcommand{\E}{\mathcal{E}}
\newcommand{\G}{\mathfrak{G}}

\newcommand{\hlift}[1]{{#1}^{\scriptscriptstyle\mathrm{H}}}
\newcommand{\hlft}[1]{{#1}^{\scriptscriptstyle\mathrm{H}}}
\newcommand{\vlift}[1]{{#1}^{\scriptscriptstyle\mathrm{V}}}
\newcommand{\hol}{\mathop{\mathrm{hol}}}

\newcommand{\I}{\mathfrak{I}}
\newcommand{\id}{\mathop{\mathrm{id}}}
\newcommand{\indic}{\mathcal{I}}
\newcommand{\J}{\mathfrak{J}}
\newcommand{\Lie}{\mathcal{L}}
\newcommand{\lie}{\Lie}
\newcommand{\K}{\mathfrak{K}}

\newcommand{\R}{\mathbb{R}}
\newcommand{\V}{\mathcal{V}}
\newcommand{\X}{\mathfrak{X}}

\newcommand{\Zbar}{\overline{Z}}

\newcommand{\chibar}{\overline{\chi}}

\newtheorem*{defn}{Definition}
\newtheorem{thm}{Theorem}
\newtheorem{cor}{Corollary}
\newtheorem{prop}{Proposition}
\newtheorem{lem}{Lemma}

\newcommand{\art}[6]{#1: #2 {\it #3\/} {\bf #4} (#5) #6}
\newcommand{\artapp}[3]{#1: #2 {\it #3\/} to appear}
\newcommand{\book}[4]{#1: {\it #2\/} (#3, #4)}
\newcommand{\inbook}[6]{#1: #2, in {\it#3\/} #4 (#5, #6)}

\RequirePackage{color}
{\egroup}

\begin{document}

\title{Holonomy of a class of bundles with fibre metrics}

\author{M.\ Crampin\\
Department of Mathematics,
Ghent University\\
Krijgslaan 281, B--9000 Gent, Belgium\\
and D.\,J.\ Saunders\\
Department of Mathematics, Faculty of Science,\\
The University of Ostrava\\
30.\ dubna 22, 701 03 Ostrava, Czech Republic}

\maketitle

\begin{abstract}\noindent
This paper is concerned with the holonomy of a class of spaces which 
includes Landsberg spaces of Finsler geometry. The methods used are 
those of Lie groupoids and algebroids as developed by Mackenzie. We 
prove a version of the Ambrose-Singer Theorem for such spaces. The 
paper ends with a discussion of how the results may be extended to 
Finsler spaces and homogeneous nonlinear connections in general.
\end{abstract}

\subsubsection*{MSC}
53C60; 53C05, 53C29, 58H05. 

\subsubsection*{Keywords}
Fibre metric; Finsler space; holonomy algebroid; holonomy groupoid;
horizontal distribution; Landsberg space; nonlinear connection.

\section{Introduction}

In most conventional accounts of the holonomy of a principal
connection (the one by Kobayashi and Nomizu \cite{KandN} has in our
view rarely if ever been bettered, so we take it as a standard
reference) attention is focussed on the holonomy group or algebra at a
single point in the base manifold.  It is, of course, proved that
holonomy groups at different points are isomorphic, and likewise
holonomy algebras:\ but these facts are treated almost in passing.
Yet the holonomy objects are pointwise representatives of global
structures, and one yearns for a theory of holonomy that adequately
reflects this fact.  It seems clear that, for example, the collection
of holonomy algebras at all points of the base manifold $M$ will form
a Lie algebra bundle over $M$ (a vector bundle whose fibres are Lie
algebras, with local trivializations which are fibrewise Lie algebra
isomorphisms with a standard Lie algebra fibre).  One would like to
know how the structure of this bundle is determined by the connection
from which it is derived.

A theory of holonomy which provides the answers to such questions has
been developed by Mackenzie \cite{MK05}.  Mackenzie's fundamental
insight is that holonomy should be seen as a branch of the theory of
Lie groupoids and Lie algebroids.  Consider the case of a connection
in a principal $G$-bundle $\pi:P\to M$.  Parallel displacement along a
piecewise smooth curve in $M$ joining points $x$ and $y$ is an
isomorphism (strictly, $G$-equivariant diffeomorphism) of the fibre
$P_x=\pi^{-1}(x)$ with the fibre $P_y$.  The collection of all such
fibre isomorphisms, for all curves $c$, is a groupoid $\Theta$ over
$M$, with source projection the initial point $x$ of $c$, target
projection the final point $y$.  The holonomy group at a point $x$ is
the vertex group at $x$ of $\Theta$, that is, the set of elements with
the same source and target $x$.  This groupoid is in fact a Lie
groupoid, and has therefore associated with it a Lie algebroid
$A\Theta$ over $M$, which is transitive; the holonomy Lie algebra
bundle is just the kernel of $A\Theta$.

It seems to us that the theory of the holonomy of nonlinear connections,
such as those that arise naturally in Finsler geometry, has so far
lacked a really satisfactory conceptual framework, and we believe that
the groupoid approach provides one.  The underlying purpose of this
paper is to initiate the development of the theory of the holonomy of
nonlinear connections from this point of view. (For a survey of 
previous work on this subject see \cite{Kozma}.) For technical reasons,
which will be discussed in some detail at the end of the main part of
the paper, we shall in fact restrict our attention for the most part
to a certain class of geometrical structures admitting nonlinear
connections, of which Landsberg spaces in Finsler geometry provide the
best-known examples.

The fundamental tensor of a Finsler space over a manifold $M$ defines
on each fibre $\TxMO$ of the slit tangent bundle a Riemannian metric.
This is evidently a particular case of a more general notion, that of
a fibre bundle with a fibre metric.  Quite a large proportion of this
paper is devoted in fact to investigating relevant aspects of the
differential geometry of such spaces in general, before we even
consider such matters as nonlinear connections and holonomy.  When we
do come to deal with holonomy we restrict our attention to a special
class of such spaces, which stand in relation to spaces with fibre
metrics in general as Landsberg spaces do to the full class of Finsler
spaces.  For technical reasons we deal only with bundles with compact
fibres (we can of course treat a Finsler structure as defined on a
bundle having compact fibres by restricting to the indicatrix bundle).
We prove a version of the Ambrose-Singer Theorem for this class of
spaces, which specifies the holonomy Lie algebra bundle in terms of
curvature.  In any account of holonomy a distinction needs to be made
between the full holonomy algebra at a point and the algebra generated
by the covariant derivatives of the curvature; the latter is in
general a proper subalgebra of the former, and indeed may differ from
point to point.  We discuss this matter in relation to Landsberg
spaces in some detail.

The paper is laid out as follows.  In the following section we give a
brief resum\'{e} of Mackenzie's theory of holonomy.  The geometry of
bundles with fibre metrics is described in Section 3.  The application
to Landsberg spaces and their holonomy occupies Section 4.  We deal
with covariant derivatives of the curvature in Section 5.  Section 6
concludes the main part of the paper with a discussion of how our
results might apply to Finsler spaces and nonlinear connections in
greater generality.  There are two appendices:\ in the first we give
the basic definitions of groupoids and Lie algebroids, for the
reader's convenience; in the second we present proofs of various
results concerning vector bundles with connection which are needed in
the main text.

\section{Groupoids and holonomy}

In this section we summarise some general results regarding groupoids
and holonomy.  The reader is referred to~\cite[Section~6.3]{MK05} for
details of the proofs.  We start by describing a connection on a
locally trivial Lie groupoid $\Omega$ over a connected base manifold
$M$, with source projection $\alpha$ and target projection $\beta$, as
a means of lifting curves in $M$ to curves in $\Omega$.
\begin{defn}
Let $\Cur(M)$ denote the set of continuous, piecewise-smooth curves $c
: [0,1] \to M$, and let $\Cur(\Omega)$ denote the corresponding set of
curves in $\Omega$.  A \emph{connection on $\Omega$} is a map $\Gamma
: \Cur(M) \to \Cur(\Omega)$, $c \mapsto c^\Gamma$, satisfying the
following properties:
\begin{enumerate}
\item $c^\Gamma(0) = 1_{c(0)}$ and, for all $t \in [0,1]$,
\[
\alpha\big(c^\Gamma(t)\big) = c(0), \qquad \beta\big(c^\Gamma(t)\big) = c(t);
\]
\item if $[a,b] \subset [0,1]$, and if $\phi : [0,1] \to [a,b]$ is a
diffeomorphism, then
\[
(c \circ \phi)^\Gamma = r_{c^\Gamma(\phi(0))^{-1}} \circ c^\Gamma \circ \phi;
\]
\item if $c$ is smooth at $t \in [0,1]$ then so is $c^\Gamma$;
\item if $c_1, c_2 \in \Cur(M)$ and, for some $t \in [0,1]$,
\[
\dot{c}_1(t) = \dot{c}_2(t)
\]
then
\[
\dot{c}^\Gamma_1(t) = \dot{c}^\Gamma_2(t);
\]
\item if $c_1, c_2, c_3 \in \Cur(M)$ and, for some $t \in [0,1]$,
\[
\dot{c}_1(t) + \dot{c}_2(t) = \dot{c}_3(t)
\]
then
\[
\dot{c}^\Gamma_1(t) + \dot{c}^\Gamma_2(t) = \dot{c}^\Gamma_3(t). 
\]
\end{enumerate} 
\end{defn}
A consequence of the properties above is that lifts are consistent
with reparametrization, so that the definition of a lift may be
extended to curves whose domains are arbitrary intervals in $\R$.

In the remainder of this paper we shall be concerned with Lie
algebroids as much as with Lie groupoids, and so we shall also need
the equivalent definition, that of an infinitesimal connection.
\begin{defn}
Let $\pi : A \to M$ be a transitive Lie algebroid, with anchor map $a
: A \to TM$.  A \emph{connection on $A$} is a vector bundle map
$\gamma : TM \to A$ over the identity on $M$ satisfying $a \circ
\gamma = \id_{TM}$.  If $\Omega$ is a locally trivial Lie groupoid
then an \emph{infinitesimal connection on $\Omega$} is a connection on
the Lie algebroid $A\Omega$.  
\end{defn}
\begin{prop}
There is a bijective correspondence between connections $\Gamma$ and
infinitesimal connections $\gamma$ on $\Omega$, given by
\[
\dot{c}^\Gamma(t) = r_{c^\Gamma(t)*} \big(\gamma(\dot{c}(t))\big). 
\] 
\end{prop}
We now consider a fixed connection $\Gamma$ on $\Omega$.  It is a
consequence of the definition that the lift of a constant curve is an
identity in $\Omega$, that the lift of a concatenation of curves is
the product of the separate lifts, and that the lift of a curve
traversed in the reverse direction is the inverse of the lift of the
original curve.  We may therefore make the following definition.
\begin{defn}
The \emph{holonomy subgroupoid} $\Theta \subset \Omega$ of the
connection $\Gamma$ is defined by
\[
\Theta = \{ c^\Gamma(1) : c \in \Cur(M) \}.
\]
For any $x \in M$ the \emph{holonomy group} of $\Gamma$ at $x$ is defined by
\[
H_x = \{ c^\Gamma(1) \in \Theta : c(0) = c(1) = x \},
\] 
and the \emph{restricted holonomy group} is the normal subgroup
$H^\circ_x \lhd H_x$ where the loops $c$ are contractible in $M$.
\end{defn}
\begin{thm}
The holonomy groupoid $\Theta$ is a Lie subgroupoid of $\Omega$.
\end{thm}
\begin{proof}
It is clear that $\Theta$ will be a subgroupoid under the operations
induced from $\Omega$.  The proof that it is a Lie subgroupoid is a
modification of the classical proof that the holonomy groups of a
connection on a principal bundle are Lie groups, as described
in~\cite[Theorem~4.2]{KandN}.

The proof starts by writing $\Omega_{x,x}$ for the subset $\{ \omega \in
\Omega : \alpha(\omega) = \beta(\omega) = x \}$; it is clear that $\Omega_{x,x}$
is both a group and a submanifold of $\Omega$; the Lie groupoid
properties of $\Omega$ then imply that $\Omega_{x,x}$ is a Lie group.
Thus $H_x$ and $H^\circ_x$ are (topological) subgroups of
$\Omega_{x,x}$.

The next step of the argument is to show that $H^\circ_x$ is
path-connected.  An intricate argument known as the Factorization
Lemma (see \cite[Appendix 7]{KandN}) shows that it is sufficient to
consider elements $c^\Gamma(1) \in H^\circ_x$ of a particular type, where the curve $c$
is known as a \emph{lasso}, and is a small
loop based at some point $y \in M$ preceded by a path from $x$ to $y$
and followed by the reverse path from $y$ to $x$.  The small loop may
be taken to lie within a single coordinate chart of $M$, and thus the
lasso $c$ may be shrunk to the single point $x \in M$, giving a path in
$H^\circ_x$ from $c^\Gamma(1)$ to the identity.  It then follows from a
standard theorem that the restricted holonomy group $H^\circ_x$, as a
path-connected subgroup of a Lie group, is itself a Lie group.

It follows from this that the full holonomy group $H_x$ is also a Lie
group. Here we use the fact that $M$ is second
countable, a consequence of the assumed connectedness of $M$ and the
standard assumption that $M$ is paracompact.  Thus the homotopy group
$\pi_1(M,x)$ is countable, and the existence of a homomorphism
$\pi_1(M,x) \to H_x / H^\circ_x$ shows that the quotient group $H_x /
H^\circ_x$ is countable.  It then follows from another standard
theorem that $H_x$ is a Lie group.

Finally, we have to show that $\Theta$ is a Lie groupoid.  The
argument here is similar to the one we give below in
Theorem~\ref{isogpoid}, using the Lie group structure of each $H_x$
and the manifold structure of $M$; more details may be found
in~\cite[Theorem~6.3.19]{MK05}.
\end{proof}
Given the connection $\Gamma$ on $\Omega$, the corresponding
infinitesimal connection $\gamma$ will be defined on the Lie algebroid
$A\Omega$.  We have seen that the holonomy groupoid $\Theta$ of
$\Gamma$ is a Lie groupoid, and so it will have a Lie algebroid
$A\Theta$ which may be identified with a Lie subalgebroid of
$A\Omega$.  We now wish to see if this particular subalgebroid may be
characterised in some way by $\gamma$.  To find such a
characterisation, we need to use the covariant differentiation and the
curvature associated with $\gamma$.
\begin{defn}
Let $\gamma$ be an infinitesimal connection on $A\Omega$.  Define, for
any vector field $X \in \X(M)$, the covariant derivative map
$\nabla^\gamma_X : \Sec(A\Omega) \to \Sec(A\Omega)$ by
\[
\nabla^\gamma_X(\phi) = [\gamma(X), \phi].
\]
Define the curvature of the connection to be the map 
$R^\gamma : \X(M) \times \X(M) \to \Sec(A\Omega)$ given by
\[
R^\gamma(X,Y) = \gamma([X,Y]) - [\gamma(X), \gamma(Y)],
\]
describing the extent to which $\gamma : \X(M) \to \Sec(A\Omega)$
fails to be a Lie algebra homomorphism. 
\end{defn}
We shall also need to consider Lie algebra sub-bundles $L$ of
$L\Omega$, the kernel of $A\Omega$.  Given such an $L$, define $A
\subset A\Omega$ as the span of certain sections of $A\Omega$,
\[
\Sec(A) = \{ \phi \in \Sec(A\Omega) : \phi - \gamma a(\phi) \in \Sec(L) \}. 
\]
\begin{prop}
If $L$ is invariant under covariant differentiation, so that
$\nabla^\gamma_X (\Sec(L)) \subset \Sec(L)$ for every $X \in \X(M)$, and
if the curvature takes its values in $L$, so that $R^\gamma(X,Y) \in
\Sec(L)$ for every $X, Y \in \X(M)$, then $A \to M$ is a Lie algebroid
with kernel $L$. 
\end{prop}

The fundamental result (\cite[Theorem~6.4.20]{MK05}) is as follows.

\begin{thm}\label{least}
There is a least Lie algebra sub-bundle $(L\Omega)^\gamma$ of
$L\Omega$ satisfying the conditions of the proposition above, and the
corresponding Lie subalgebroid $(A\Omega)^\gamma$ of $A\Omega$
satisfies
\[
(A\Omega)^\gamma = A\Theta.
\]
\end{thm}

This description may appear somewhat formal, so let us explain
how it works in a familiar situation, that of a linear connection on a
manifold $M$.  We may think of such a connection in various ways:\ as
a covariant derivative; as a law of parallel transport; or as a
horizontal distribution on the tangent bundle $TM$, spanned by local
vector fields
\[
\hlift{\left(\vf{x^i}\right)}=\vf{x^i}-\conn jik u^k\vf{u^j},
\]
in the usual notation.

We first remark that if a vector field on $TM$ takes the 
form
\[
\xi^i\vf{x^i}+\eta^j_ku^k\vf{u^j}
\]
in terms of some coordinates $x^i$ on $M$, where $\xi^i$ and
$\eta^j_k$ are functions of the $x^i$ alone, then (as may easily be
checked) it takes the same form when new coordinates are chosen for
$M$.  We call such a vector field projectable fibre-linear; evidently 
it projects onto $\xi^i\partial/\partial x^i$ on $M$.

To start the groupoid description of the holonomy of a linear
connection we take for the ambient groupoid $\Omega$ the Lie groupoid
whose elements with source $x$ and target $y$ are the (linear) fibre
isomorphisms $T_xM\to T_yM$.  The elements of the Lie algebroid
$A\Omega$ over $x$ are projectable fibre-linear vector fields along
$T_xM\to TM$:\ such a vector field is defined on, but is not
necessarily tangent to, the fibre $T_xM$, and takes the form displayed
above, where now $\xi^i$ and $\eta^j_k$ are constants.  Notice that
{\em points\/} of $(A\Omega)_x$ are {\em vector fields\/} along
$T_xM\to TM$.  The space of projectable fibre-linear vector fields
along $T_xM\to TM$ is finite-dimensional, and $A\Omega$ is indeed a
vector bundle.  Sections of $A\Omega\to M$ are projectable
fibre-linear vector fields on $TM$, and the Lie algebroid bracket is
just the ordinary bracket of vector fields on $TM$ (of course the
bracket of projectable fibre-linear vector fields is projectable
fibre-linear).  The anchor is projection, and the kernel $L\Omega$
consists of vertical fibre-linear vector fields on $TM$.

We turn next to the connection on $\Omega$ corresponding to the given
linear connection.  Let $c\in\Cur(M)$:\ we have to define
$c^\Gamma\in\Cur(\Omega)$.  Thus for each $t\in[0,1]$, $c^\Gamma(t)$
should be a fibre isomorphism $T_{c(0)}M\to T_{c(t)}M$.  Parallel
translation along $c$ from $c(0)$ to $c(t)$ is such a fibre
isomorphism, and if we take this for $c^\Gamma(t)$ we see it satisfies
all the requirements for a connection on $\Omega$. The corresponding 
element of the holonomy groupoid is parallel translation along $c$ 
from $c(0)$ to $c(1)$.

The corresponding infinitesimal connection $\gamma : TM\to A\Omega$ 
is just the horizontal lift:\ for $v\in T_xM$, 
$v=v^i\partial/\partial x^i$
\[
\gamma(v)=\hlift{v}=v^i\left(\vf{x^i}-\conn jik u^k\vf{u^j}\right).
\]
Of course we have to ensure that $\gamma(v)$ is an element of
$(A\Omega)_x$, that is, a projectable fibre-linear vector field
along $T_xM\to TM$ --- which indeed $\hlift{v}$ is.

Now any projectable fibre-linear vector field on $TM$ (section of
$A\Omega\to M$ in other words) can be written uniquely in the form
$\hlift{X}+\vlift{S}$ where $X$ is a vector field and $S$ a type
$(1,1)$ tensor field on $M$, and the `vertical lift' $\vlift{S}$ of
$S$ is given by
\[
\vlift{S}=S^i_ju^j\vf{u^i}.
\]
(One needs the connection to make the vertical part a tensor:\
$\eta^i_j$ in the previous incarnation is not a tensor.)  We can
therefore identify $A\Omega$ with $TM\oplus T^1_1M$.  By
straightforward calculations
\begin{align*}
[\hlift{X},\hlift{Y}]&=\hlift{[X,Y]}-\vlift{R(X,Y)};\\\mbox{}
[\hlift{X},\vlift{T}]&=\vlift{(\nabla_XT)};\\\mbox{}
[\vlift{S},\vlift{T}]&=-\vlift{\{S,T\}},
\end{align*} 
where $R$ is the curvature tensor of the linear connection (so
$R(X,Y)$ is a type $(1,1)$ tensor), and $\{S,T\}$ is the commutator of
$S$ and $T$ (considered as matrices).  So the bracket of sections of
$A\Omega$ is given by
\[
[X\oplus S,Y\oplus T]=
[X,Y]\oplus\left(-R(X,Y)+\nabla_XT-\nabla_YS-\{S,T\}\right).
\] 
In particular, we can identify $L\Omega$ with $T^1_1M$, with bracket
the negative of the commutator.  The objects of interest now are the
Lie algebra sub-bundles of $T^1_1M$ whose section spaces are closed
under covariant derivative and contain all curvature tensors $R(X,Y)$;
the holonomy Lie algebra bundle is the least such.

\section{Fibre metrics}

We now turn to the definition and properties of fibre metrics:\ but
first we prove a result which has an important role to play at several
points in this section.

\begin{prop}\label{interval}
Let $\pi : E \to M$ be a fibre bundle whose standard fibre is compact.
Let $Z$ be any vector field on $E$, with flow $\chi_t$.  Then for any
$x\in M$ there is an open interval $I$ containing 0 such that for all
$u\in E_x$, $\chi_t(u)$ is defined for all $t\in I$. If $Z$ is
projectable to a vector field $\Zbar \in \X(M)$, its flow $\chi_t$ is
projectable to the flow $\chibar_t$ of $\Zbar$; then $\chibar_t(x)$
is also defined for all $t\in I$.
\end{prop}

\begin{proof}
For each $u\in E_x$ there is an open neighbourhood $U_u$ of
$u$ in $E$ and an open interval $I_u$ containing 0 such that
$\chi_t(v)$ is defined for all $t\in I_u$ and $v\in U_u$, and in
particular for all $v\in U_u\cap E_x$.  The open sets $ U_u\cap E_x$
cover $E_x$, so we can find a finite subcover:\ let $I$ be the
intersection of the corresponding open intervals $I_u$.  Then $I$,
being the intersection of a finite number of open intervals, is open,
and contains 0, and for every $v\in E_x$, $\chi_t(v)$ is defined for
all $t\in I$.
\end{proof}

Let $\pi : E \to M$ be a fibre bundle, and $V\pi\to E$ the vertical
sub-bundle of $TE$.  By a type $(0,2)$ fibre tensor on $E$ we mean a
smooth bilinear map $V\pi \times_E
V\pi \to \R$.  If $g$ is a type $(0,2)$ fibre tensor on $E$ then
$g_x$, its restriction to $E_x=\pi^{-1}(x)$, is a type $(0,2)$ tensor
field on $E_x$.  A fibre metric on $E$ is a type $(0,2)$ fibre tensor
for which $g_x$ is a Riemannian metric on $E_x$ for all $x$.

Let $Z$ be a projectable vector field on $E$; then for any vector
field $V$ on $E$ which is vertical over $M$, $[Z,V]$ is vertical.  So
for any vertical vector fields $V$, $W$ on $E$ we may set
\[
\Lie_Zg(V,W)= Z\big( g(V,W) \big) - g \big( [Z,V], W \big) 
- g \big(V, [Z,W] \big).
\]
Clearly, $\Lie_Zg(V,W)$ is $C^\infty(E)$-bilinear in $V$ and $W$, so
this formula defines a type $(0,2)$ fibre tensor $\Lie_Zg$, which is
evidently symmetric.

The operator $\Lie_Z$ is a form of Lie derivative, as can be seen in
two different ways.  In the first place, we can think of $V\mapsto
[Z,V]$ as a Lie derivative of vertical vector fields (identifying the
Lie derivative and Lie bracket in the usual way); then the formula for
$\Lie_Zg$ above can be written
\[
\Lie_Zg(V,W)=\lie_Z\big( g(V,W) \big) 
- g \big( \lie_ZV, W \big) - g \big( V, \lie_ZW \big), 
\]
which mimics the usual method of extending the Lie derivative from
vector fields to tensor fields. This formulation makes it clear that 
to evaluate $\Lie_Zg(V,W)$ at any $u\in E$ we need consider only the 
values of $V$ and $W$ (and of course $g$) along the integral curve of 
$Z$ through $u$. 

Secondly, when the fibres of $E$ are compact we can interpret the
formula in terms of flows, as follows.  Since $Z$ is projectable, to
$\Zbar \in \X(M)$, its flow $\chi_t$ is projectable to the flow
$\chibar_t$ of $\Zbar$.  Given $x \in M$, by
Proposition~\ref{interval} $\chi_t(u)$ is defined for all $u\in E_x$
for $t$ in some open interval $I$ containing 0.  Denote the
restriction of $\chi_t$, $t\in I$, to $E_x$ by $\chi_{x;t}$; thus
$\chi_{x;t}$ is a diffeomorphism of $E_x$ with $E_{\chibar_t(x)}$.
The pullback $\chi_{x;t}^*(g_{\chibar_t(x)})$ is defined for any $t\in
I$, and is another symmetric type $(0,2)$ tensor field on $E_x$.  We
claim that
\[
\dte{\chi_{x;t}^* \big( g_{\chibar_t(x)} \big)} = 
\left(\lie_Zg\right)_x.
\]
To establish the claim, choose vertical vector fields $V_x, W_x$ on 
$E_x$, and extend them along $x_t = \chibar_t(x)$ by Lie transport, 
so that $V_{x_t}=\chi_{x;t*}V_x$. Then
\[
\left(\chi_{x;t}^* g_{x_t}\right)(V_x,W_x)
=g_{x_t}\left(\chi_{x;t*}V_x,\chi_{x;t*}W_x\right)
=g_{x_t}\left(V_{x_t},W_{x_t}\right)
=g(V,W)(x_t).
\]
Thus
\[
\dte{\chi_{x;t}^* g_{x_t}}(V_x,W_x) = 
Z_x\left(g(V,W)\right).
\]
But $\lie_ZV=\lie_ZW=0$ by construction, so in this case
\[
\dte{\chi_{x;t}^* g_{x_t}}(V_x,W_x) = 
\left(\lie_Zg\right)_x(V_x,W_x).
\]
But each of $d/dt(\chi_{x;t}^*g_{x_t})_{t=0}$ and
$\left(\lie_Zg\right)_x$ is a tensor field on $E_x$, and so they are
equal (as tensors).

Of course if $Z_1$ and $Z_2$ are projectable so is $[Z_1,Z_2]$. It 
is easy to see that 
\[
\lie_{Z_1}\left(\lie_{Z_2}g\right)-\lie_{Z_2}\left(\lie_{Z_1}g\right)
=\lie_{[Z_1,Z_2]}g.
\]

So far, so fairly predictable; but now an unexpected feature 
emerges. For any $f\in C^\infty(M)$,
\[
\Lie_{fZ}g(V,W)=
fZ\big( g(V,W) \big) - g \big( [fZ,V], W \big) - g \big( V, [fZ,W] \big)
= f\Lie_Zg(V,W),
\]
since $Vf=Wf=0$.  That is to say, $\Lie_{Z}g(V,W)$ is
$C^\infty(M)$-linear in $Z$.  Thus so far as its dependence on $Z$ is
concerned, $\lie_Zg$ behaves like a covariant derivative rather than a
Lie derivative.  In particular, $(\lie_Zg)_x$ depends only on the
value of $Z$ on $E_x$.  To put things another way, for any vector
field $Z_x$ along $E_x\to E$ which is projectable to
$T_xM$, there is a well-defined symmetric type $(0,2)$ tensor field
$\lie_{Z_x}g$ on $E_x$, given by $\lie_{Z_x}g=(\lie_Zg)_x$ where $Z$
is any vector field on $E$ defined in a neighbourhood of $E_x$ which
agrees with $Z_x$ on $E_x$.  This feature is clear also from the
coordinate representation of $\lie_Zg$.  Take coordinates $x^i$ on $M$
and $u^a$ on the fibre; let
\[
Z=\xi^i\vf{x^i}+\eta^a\vf{u^a},\qquad
g\left(\vf{u^a},\vf{u^b}\right)=g_{ab},
\]
where $\xi^i=\xi^i(x)$; then
\[
(\lie_Zg)_{ab}=\xi^i\fpd{g_{ab}}{x^i}+\eta^c\fpd{g_{ab}}{u^c}
+g_{cb}\fpd{\eta^c}{u^a}+g_{ac}\fpd{\eta^c}{u^b}.
\]
The point to note is that no derivatives of the components of $Z$ with
respect to the $x^i$ appear on the right-hand side.  It is also clear
from this formula, if it wasn't already, that if $Z_x$ is actually
vertical then $\lie_{Z_x}g$ is just the ordinary Lie derivative of 
$g_x$ considered as a tensor field on $E_x$.

We call $\lie_Zg=0$ the isometry equation for (projectable) vector 
fields, and $\lie_{Z_x}g=0$ the isometry equation at $x$. Every 
vector field solution of the isometry equation gives rise to a 
solution of the isometry equation at $x$, but it is not necessarily 
the case that a solution of the isometry equation at $x$ can be 
extended to a vector field solution, even locally. 

Suppose that $Z$ is a solution of the isometry equation; denote its
flow by $\chi_t$, and the flow of its projection $\Zbar$ by
$\chibar_t$.  We know from Proposition~\ref{interval} that for any
$x\in M$, $\chi_{x;t}$ is well-defined for all $t$ in some open
interval $I$ containing 0.  It follows from the fact that $\lie_Zg=0$,
by a standard argument, that
\[
\frac{d}{dt}\left(\chi_{x;t}^* g_{x_t}\right)=0
\]
for all $t\in I$, and hence that $\chi_{x;t}^* g_{x_t}=g_x$.  Thus a
solution of the isometry equation for vector fields is the
infinitesimal generator of fibre isometries of the fibre metric.  We
call such a vector field an infinitesimal fibre isometry of the fibre
metric, and denote the space of infinitesimal fibre isometries by
$\J$.  Then $\J$ is a $C^\infty(M)$-module of vector fields on $E$
which is closed under bracket.

The solutions of the isometry equation at $x$ form a vector space over
$\R$ which we denote by $\J^x$.  Those solutions which are vertical
are simply infinitesimal isometries, or Killing fields, of the metric
$g_x$.  Since a solution of the isometry equation at $x$ is
projectable to $T_xM$, the linear map $\pi_*$ restricts to a linear
map $\J^x\to T_xM$, and we see that the solution space $\J^x$ is of
finite dimension at most $m+\onehalf n(n-1)$ where $m=\dim M$ and
$n=\dim E_x$.  The kernel of $\pi_*|_{\J^x}$, which we denote by
$\K^x$, is the space of infinitesimal isometries of $g_x$, and in
particular is a Lie algebra (under Lie bracket), not just a vector
space.

It is tempting to think of $\J$ as consisting of the sections of a
vector bundle over $M$ whose fibre at $x$ is $\J^x$, but this will not
normally be permissible:\ for example, there is no guarantee that the
spaces $\J^x$ at different points $x$ are isomorphic. We next
describe a situation in which this difficulty does not arise.

Consider a fibre bundle $\pi : E \to M$ whose standard fibre is
compact and whose base is connected, which is equipped with a fibre
metric $g$ such that the map $\pi_*:\J\to\X(M)$, taking each
infinitesimal fibre isometry to its projection on $M$,
is surjective. In such a case we say that $\J$ is transitive. 

\begin{prop}
Let $\pi : E \to M$ be a fibre bundle with compact standard fibre and
connected base, equipped with a fibre metric $g$ such that $\J$ is
transitive.  The fibres of $E$, considered as Riemannian manifolds,
are pairwise isometric.
\end{prop}

\begin{proof}
Let $x$ and $y$ be points of $M$ that both lie on an integral
curve of some vector field $X$.  We show that $E_x$ and $E_y$ are
isometric.  Let $\varphi_t$ be the flow of $X$, and suppose that
$y=\varphi_s(x)$ (without loss of generality we may assume that
$s>0$).  By assumption there is a vector field $\tilde{X}$ on $E$ such
that $\tilde{X}\in\J$ and $\pi_*\tilde{X}=X$.  Let $\tilde{\varphi}_t$
be the flow of $\tilde{X}$:\ then
$\pi\circ\tilde{\varphi}_t=\varphi_t$, and for any $z\in M$ there is
an open interval $I$ containing 0 such that $\tilde{\varphi}_t$ is an
isometry of $E_z$ with $E_{\varphi_t(z)}$ for all $t\in I$.  This
holds in particular for $z=\varphi_r(x)$ for all $r\in[0,s]$.  So we
have a covering of $[0,s]$ by open intervals say $(r-\delta_r,r+\delta_r)$
on each of which $\tilde{\varphi}_t(E_{\varphi_r(x)})$ is defined.
From this covering we can extract a finite subcovering.  Then using
the one-parameter group property we see that in fact
$\tilde{\varphi}_t(E_x)$ is defined for all $t\in[0,s]$, and therefore
$\tilde{\varphi}_s:E_x\to E_y$ is an isometry (and indeed
$\tilde{\varphi}_r:E_x\to E_{\varphi_r(x)}$ is an isometry for all
$r\in[0,s]$).  Now any pair of points in $M$ can be joined by a
piecewise smooth curve which is made up of segments each of which is
part of the integral curve of some vector field; the general result
follows.
\end{proof}

Since the fibres are all isometric, we can choose a Riemannian 
manifold isometric to each of them, for example any one fibre $E_x$ with 
its metric $g_x$. We denote this representative Riemannian manifold 
by $\E$.

\begin{prop}\label{isotriv}
Let $\pi : E \to M$ be a fibre bundle with compact standard fibre and
connected base, equipped with a fibre metric $g$ such that $\J$ is
transitive.  About every point of $M$ there is a neighbourhood $U$ and
a local trivialization $\tau: U\times \E\to E|_U$, compatible with the
smooth bundle structure, such that for all $x\in U$, the map
$\tau_x:\E\to E_x$ given by $\tau_x=\tau(x,\cdot)$ is an isometry.
\end{prop}

\begin{proof}
Fix a point in $M$ and take a coordinate neighbourhood $U$ with this
point as origin; we shall denote it by $o$.  Without loss of generality
we may assume that the image of $U$ is the open unit ball in $\R^m$,
and we may further assume that $U$ is contained in some neighbourhood
over which $E$ is locally trivial.  Since $\J$ is a
$C^\infty(M)$-module, with each coordinate field
$X_i=\partial/\partial x^i$ we may associate a vector field
$\tilde{X}_i$ on $E$ such that $\tilde{X_i}\in\J$ and
$\pi_*\tilde{X_i}=X_i$ on $U$.  For each $x\in U$ let $r_{x}$ be the
ray joining the origin to $x$ with respect to the coordinates:\ it is
an integral curve of $x^i(x)X_i$, where the $x^i(x)$ are the coordinates
of $x$ (which of course are to be treated as constants).  Since $\J$
is a vector space over $\R$, $x^i(x)\tilde{X}_i\in\J$.  Let
$\sigma_x:E_o\to E_x$ be the isometry determined by
$x^i(x)\tilde{X}_i$.  Let $u^a$ be fibre coordinates on $E$ over $U$.
Then
\[
\tilde{X}_i=\vf{x^i}+N_i^a\vf{u^a}
\]
for certain functions $N_i^a$ on $E|_U$, and so
\[
x^i(x)\tilde{X}_i=x^i(x)\left(\vf{x^i}+N_i^a\vf{u^a}\right).
\]
Thus for any $u\in E_o$, $u=(0,u^a)$, the fibre coordinates of
$\sigma_x(u)$ are determined as the solution of the system of ordinary
differential equations
\[
\dot{u}^a=x^i(x)N_i^a(tx^j(x),u^b),\quad u^a(0)=u^a,
\]
in which the coordinates $x^i(x)$ play the role of parameters.  But the
solutions of such equations depend smoothly on parameters, so
$\sigma_x(u)$ depends smoothly on the coordinates $(x^i,u^a)$.  Now
let $\phi:\E\to E_o$ be an isometry.  We define $\tau: U\times \E\to
E|_U$ by $\tau(x,v)=\sigma_x(\phi(v))$ for $v\in\E$.  Then $\tau$ is a
local trivialization, compatible with the smooth bundle structure,
such that for all $x\in U$, $\tau_x=\sigma_x\circ\phi:\E\to E_x$ is an
isometry.
\end{proof}

For any bundle $E\to M$ with fibre metric $g$ the set of isometries 
between fibres is evidently a groupoid, which we call the 
fibre-isometry groupoid (the precise definition is the first 
paragraph of the proof of the following theorem).

\begin{thm}\label{isogpoid}
The fibre-isometry groupoid of a fibre bundle $\pi : E \to M$ with
compact standard fibre and connected base, equipped with a fibre
metric $g$ such that $\J$ is transitive, is a locally trivial Lie groupoid.  The
corresponding Lie algebroid has the spaces $\J^x$ for its fibres and
the space $\J$ as its sections.
\end{thm}

\begin{proof}
Given the bundle $\pi : E\to M$, we let $\Omega_{y,x}$ denote the set
of isometries from the fibre $E_x$ to the fibre $E_y$, and let $\Omega
= \bigcup_{x,y \in M} \Omega_{y,x}$.  If $\theta \in \Omega_{y,x}$ put
$\alpha(\theta) = x$ and $\beta(\theta) = y$; this defines maps $\alpha,
\beta : \Omega \to M$.  Let $1_x : E_x \to E_x$ denote the identity
isometry; then $1 : x \mapsto 1_x$ maps $M$ to $\Omega$.  Finally, let
the partial multiplication be composition of maps, and the inverse be
the usual inverse of maps.  It is immediate that, with these
definitions, $\Omega$ is a groupoid.

To give $\Omega$ a smooth structure we modify the technique used in
\cite[Example~1.1.12]{MK05}.  Let $\{U_\lambda\}$, for $\lambda$ in
some indexing set $\Lambda$, be a covering of $M$ by sets of the kind
defined in Proposition~\ref{isotriv}. Let $\tau_\lambda :
U_\lambda \times \E \to \eval{E}{U_\lambda}$ be the corresponding
isometric local trivialisation of $\eval{E}{U_\lambda}$, and for 
$x\in U_\lambda$ set $\tau_{\lambda,x}=\tau_\lambda(x,\cdot):\E\to E_x$.  Let
\[
\Omega_{U_\mu, U_\lambda} = \bigcup_{x \in U_\lambda, y \in U_\mu} \Omega_{y,x}
\]
and define
\[
\psi_{\mu, \lambda} : 
\Omega_{U_\mu, U_\lambda} \to U_\mu \times \G(\E) \times U_\lambda,
\]
where $\G(\E)$ is the isometry group of $\E$ (a Lie group), as follows.
If $\theta \in \Omega_{U_\mu, U_\lambda}$ then $\theta
\in \Omega_{y,x}$ for some $x \in U_\lambda$ and $y \in U_\mu$.  Thus
$\big(\tau_{\mu,y}\big)^{-1} \circ \theta \circ \tau_{\lambda,x}$
is a map $\E \to \E$, which is evidently an isometry, say $G \in
\G(\E)$.  We may therefore put
\[
\psi_{\mu, \lambda}(\theta) = (y,G,x),
\]
and it is straightforward to see that this defines a bijection.

To show that $\Omega$ has a smooth structure, we need to show that
if $\psi_{\mu_2, \lambda_2} \circ \big( \psi_{\mu_1, \lambda_1}
\big)^{-1}$ has a non-empty domain then it is smooth.  The non-empty
domain condition is that
\[
\big( \psi_{\mu_1, \lambda_1} \big)^{-1}(U_{\mu_1} \times \G(\E) 
\times U_{\lambda_1}) \; \cap \; \eval{\Omega}{U_{\mu_2}, U_{\lambda_2}} 
\ne \varnothing,
\]
which translates as
\[
\Omega_{U_{\mu_1}, U_{\lambda_1}} \; \cap \; 
\Omega_{U_{\mu_2}, U_{\lambda_2}} \ne \varnothing,
\]
in other words that $U_{\lambda_1} \cap U_{\lambda_2} \ne \varnothing$
and $U_{\mu_1} \cap U_{\mu_2} \ne \varnothing$.  If this condition
holds then the domain of $\psi_{\mu_2, \lambda_2} \circ \big(
\psi_{\mu_1, \lambda_1} \big)^{-1}$ is
\[
(U_{\mu_1} \cap U_{\mu_2}) \times \G(\E) \times (U_{\lambda_1} 
\cap U_{\lambda_2}).
\]
If $(y,G,x)$ is in the domain then 
$\big(\psi_{\mu_1,\lambda_1}\big)^{-1}(y,G,x)$ is the isometry 
$\tau_{\mu_1,y}\circ G\circ\big(\tau_{\lambda_1,x}\big)^{-1}:E_x\to E_y$. Thus 
\begin{align*}
&\psi_{\mu_2, \lambda_2} \circ \big( \psi_{\mu_1, \lambda_1} \big)^{-1}(y,G,x)
 = \psi_{\mu_2, \lambda_2} 
\left(\tau_{\mu_1,y}\circ G\circ\big(\tau_{\lambda_1,x}\big)^{-1}\right)\\
&\qquad 
=\left(y,\big(\big(\tau_{\mu_2,y}\big)^{-1}\circ\tau_{\mu_1,y}\big)\circ G
\circ\big(\big(\tau_{\lambda_1,x}\big)^{-1}\circ\tau_{\lambda_2,x}\big),x\right).
\end{align*}
Note that
$\big(\tau_{\lambda_1,x}\big)^{-1}\circ\tau_{\lambda_2,x}\in\G(\E)$,
and also $\big(\tau_{\mu_2,y}\big)^{-1}\circ\tau_{\mu_1,y}\in\G(\E)$,
so the central term above is an element of $\G(\E)$, as required.
Moreover,
$x\mapsto\big(\tau_{\lambda_1,x}\big)^{-1}\circ\tau_{\lambda_2,x}$ is
the transition function $U_{\lambda_1}\cap U_{\lambda_2}\to\G(\E)$ for
the bundle structure on $\pi:E\to M$ defined in
Proposition~\ref{isotriv}, and likewise $y\mapsto
\big(\tau_{\mu_1,y}\big)^{-1}\circ\tau_{\mu_2,y}$ is the transition
function $U_{\mu_1}\cap U_{\mu_2}\to\G(\E)$.  So $\psi_{\mu_2,
\lambda_2} \circ\big( \psi_{\mu_1, \lambda_1}\big)^{-1}$ is
constructed from smooth maps by smooth operations, and therefore is
smooth.

Now $U_\lambda,U_\mu\subset M$ are coordinate patches, by assumption.
Taking a coordinate patch $V\in\G(\E)$ we obtain a coordinate patch
$\big(\psi_{\mu, \lambda}\big)^{-1}(U_\mu \times V \times U_\lambda)$
on $\Omega$.  The transition functions for such patches are smooth as a
consequence of the immediately preceding result about $\psi_{\mu_2,
\lambda_2} \circ \big( \psi_{\mu_1, \lambda_1} \big)^{-1}$.  Thus
$\Omega$ is a differentiable manifold of dimension $2m+d$, $d=\dim
\G(\E)$.

It is now straightforward to show that the structure maps of the
groupoid are smooth and that the source and target projections are
surjective submersions, both separately and as the pair $(\alpha,\beta)$, 
so we conclude that $\Omega$ is a locally trivial Lie groupoid
of dimension $2m+d$. 

We turn now to the Lie algebroid of $\Omega$, which we denote by
$A\Omega$.  We can identify points of $A\Omega$ over $x\in M$ with
tangent vectors at $t=0$ to curves $\kappa$ in the $\alpha$ fibre of
$\Omega$ over $x$ such that $\kappa(0)=1_x$.  For such a curve
$\alpha(\kappa(t))=x$, while $\beta(\kappa(t))=\overline{\kappa}(t)$
say is a smooth curve in $M$; so $\kappa(t)$ is an isometry of $E_x$
with $E_{\overline{\kappa}(t)}$ and $\kappa(0)$ is the identity map of
$E_x$.  Since $\kappa(t)$ is an isometry, for every $u\in E_x$
\[
\kappa(t)^*g_{\kappa(t)u}=g_u,
\]
or in terms of fibre coordinate fields
\[
g_{\kappa(t)u}\left(\kappa(t)_*\vf{u^a},\kappa(t)_*\vf{u^b}\right)=g_{ab}(u).
\]
Let us set $\kappa(t)u=(\overline{\kappa}^i(t),\kappa^a(t,u))$; then 
\[
\kappa(t)_*\left(\eval{\vf{u^a}}{u}\right)=\fpd{\kappa^b}{u^a}(t,u)
\eval{\vf{u^b}}{\kappa(t)u},
\]
whence
\[
g_{\kappa(t)u}\left(\kappa(t)_*\vf{u^a},\kappa(t)_*\vf{u^b}\right)=
\fpd{\kappa^c}{u^a}(t,u)\fpd{\kappa^d}{u^b}(t,u)g_{cd}(\kappa(t)u).
\]
Note that since $\kappa(0)$ is the identity, 
\[
\fpd{\kappa^c}{u^a}(0,u)=\delta^c_a.
\]
The tangent vector $Z_u$ to the curve $t\mapsto\kappa(t)u$ at $t=0$ is
\[
\frac{d\bar{\kappa}^i}{dt}(0)\vf{x^i}+\fpd{\kappa^a}{t}(0,u)\vf{u^a}
=\xi^i\vf{x^i}+\eta^a(u)\vf{u^a}
\]
say. Then $Z$ is a vector field along $E_x\to E$, 
which projects onto a vector at $x\in M$, namely the initial tangent 
vector to $\overline{\kappa}$. By differentiating the isometry condition
\[
\fpd{\kappa^c}{u^a}(t,u)\fpd{\kappa^d}{u^b}(t,u)g_{cd}(\kappa(t)u)=
g_{ab}(u)
\]
with respect to $t$ and setting $t=0$ we obtain
\[
\xi^i\fpd{g_{ab}}{x^i}+\eta^c\fpd{g_{ab}}{u^c}
+g_{cb}\fpd{\eta^c}{u^a}+g_{ac}\fpd{\eta^c}{u^b}=0.
\]
This is just the condition for $Z$ to belong to $\J^x$.

Conversely, suppose that $Z\in\J^x$.  Let $X$ be any vector field
on $M$ such that $X_x=\pi_*Z$.  Let $\tilde{X}$ be a vector field
on $E$ such that $\tilde{X}\in\I$ and $\pi_*\tilde{X}=X$.  Then
$\tilde{X}_x-Z$ is a vertical vector field on $E_x$ belonging to
$\J^x$, and is therefore an infinitesimal isometry of $g_x$.  Let
$\tilde{\varphi}_t$ be the flow of $\tilde{X}$.  By
Proposition~\ref{interval} there is an open interval containing 0 such
that $\tilde{\varphi}_t:E_x\to E_{\varphi_t(x)}$ is an isometry for
all $t\in I$, where $\varphi_t=\pi\circ\tilde{\varphi}_t$.  Set
$V(t)=\tilde{\varphi}_{t*}(\tilde{X}_x-Z)$:\ since
$\tilde{\varphi}_t$ is an isometry, $V(t)$ is a vertical vector
field which is an infinitesimal isometry of $g_{\varphi_t(x)}$.  
Consider $\tilde{X}+V$:\ this is a vector field over the curve
$\varphi_t(x)$ which projects onto $X$ and belongs to
$\J^{\varphi_t(x)}$ for all $t\in I$.  It generates a curve of
isometries $E_x\to E_{\varphi_t(x)}$, defined on some open interval
containing 0 (possibly smaller than $I$), whose tangent at $t=0$ is
just $Z$, as required.
\end{proof}

When the conditions of Theorem~\ref{isogpoid} apply we shall denote
the Lie algebroid of the fibre-isometry groupoid by $\I$.  Then $\I$
is a vector bundle over $M$, whose fibre $\I_x$ over $X\in M$ consists
of the vector fields along $E_x\to E$ which satisfy the
isometry equation at $x$.  We denote by $\Sec(\I)$ the
$C^\infty(M)$-module of sections of $\I$; elements of $\Sec(\I)$ are
vector fields on $E$ which satisfy the isometry equation.  The anchor
is $\pi_*|_\I$, and the bracket is just the bracket of vector fields
on $E$, restricted of course to elements of $\Sec(\I)$.  We denote
the kernel by $\K\to M$; it is a Lie algebra bundle, and for each
$x\in M$ the fibre $\K_x$ is the Lie algebra of infinitesimal
isometries of the Riemannian manifold $(E_x,g_x)$. 

We have the following short exact sequence of vector bundles over $M$:
\[
0\to\K\to\I\to TM\to 0.
\]
We now specialise further to the case where this sequence splits; that
is to say, we suppose that there is a linear bundle map $\gamma:TM\to
\I$ over the identity of $M$ such that $\pi_*\circ\gamma=\id_{TM}$.
Such a splitting is an infinitesimal connection on $\I$, in the
terminology of Section~2. There is a corresponding connection 
$\Gamma$ on the fibre-isometry groupoid. We now set out to identify its 
holonomy Lie algebra bundle.

\begin{prop}\label{connK}
Suppose we have an infinitesimal connection $\gamma$ on $\I$. For any 
$X\in\X(M)$ and $V\in\Sec(\K)$ set  
\[
\nabla^\gamma_XV=[\gamma(X),V].
\]
Then $\nabla^\gamma$ is a covariant differentiation operator on
$\Sec(\K)$.  Furthermore, $\nabla^\gamma_X$ is a derivation of the
bracket:
\[
\nabla^\gamma_X[V,W]=[\nabla^\gamma_XV,W]+[V,\nabla^\gamma_XW].
\]
\end{prop}

\begin{proof}
Since $\gamma(X)\in\Sec(\I)$ and $\Sec(\I)$ is closed under 
bracket, $\nabla^\gamma_XV\in\Sec(\I)$. But  
$\nabla^\gamma_XV$ is clearly vertical, so 
$\nabla^\gamma_XV\in\Sec(\K)$. Evidently $\nabla^\gamma_XV$
is $\R$-linear in both arguments, and 
for $f\in C^\infty(M)$, 
\[
\nabla^\gamma_{fX}V=f\nabla^\gamma_XV,\qquad
\nabla^\gamma_X(fV)=f\nabla^\gamma_XV+X(f)V.
\]
The fact that $\nabla^\gamma_X$ is a derivation of the bracket 
follows from the Jacobi identity.
\end{proof} 

The conclusion that $\nabla^\gamma_XV\in\Sec(\K)$ for all
$X\in\X(M)$ and $V\in\Sec(\K)$ is abbreviated to
$\nabla^\gamma(\Sec(\K))\subset\Sec(\K)$.

A connection for which the derivation property is satisfied is called
a Lie connection.  It is a consequence of the fact that the connection
is Lie that the parallel translation operator $\tau^\gamma_c$
corresponding to $\nabla^\gamma$ along a curve $c$ in $M$ from $x$ to
$y$ is an isomorphism of Lie algebras $\K_x\to\K_y$ (see
Proposition~\ref{Lie} in Appendix~2).

For any $X,Y\in\X(M)$ we set
\[
R^\gamma(X,Y)=\gamma([X,Y])-[\gamma(X),\gamma(Y)].
\]
Then $R^\gamma(X,Y)$ is a vertical vector field on $E$ which depends
$C^\infty(M)$-linearly on $X$ and $Y$, so for any $x\in M$ and any
$v,w\in T_xM$, $R^\gamma_x(v,w)$ is a well-defined vertical vector
field on the fibre $E_x$.  We call $R^\gamma$ the curvature of
$\gamma$ and we call $R^\gamma_x(v,w)$ a curvature vector field at
$x$.

\begin{prop}\label{curvK}
For all $X,Y\in\X(M)$, $R^\gamma(X,Y)\in\Sec(\K)$.
\end{prop}

\begin{proof}
Since $\Sec(\I)$ is closed under bracket,
$[\gamma(X),\gamma(Y)]\in\Sec(\I)$.  Thus both $\gamma([X,Y])$ and
$[\gamma(X),\gamma(Y)]$ belong to $\Sec(\I)$; their difference,
which is vertical, therefore belongs to $\Sec(\K)$.
\end{proof}

\begin{cor}
For every $x\in M$ and $v,w\in T_xM$, $R^\gamma_x(v,w)\in\K_x$.
\end{cor}

We can also compute the curvature of $\nabla^\gamma$; we obtain
\begin{align*}
\nabla^\gamma_X\nabla^\gamma_YV-\nabla^\gamma_Y\nabla^\gamma_XV
-\nabla^\gamma_{[X,Y]}V
&=[V,\gamma([X,Y])-[\gamma(X),\gamma(Y)]]\\
&=[V,R^\gamma(X,Y)].
\end{align*}

Due to the existence of the splitting we have $\I\equiv TM\oplus\K$.  Thus
sections of $\I\to M$ are vector fields on $E$ of the form
$\gamma(X)+V$ with $V\in\Sec(\K)$.  In terms of the direct sum
decomposition the algebroid bracket is given by
\[
[X\oplus V,Y\oplus W]=[X,Y]\oplus(-R^\gamma(X,Y)+\nabla^\gamma_XW
-\nabla^\gamma_YV+[V,W]).
\]
Every term in the second component on the right belongs to
$\Sec(\K)$.

By Theorem~\ref{least}, the
holonomy Lie algebra bundle is the least Lie algebra sub-bundle
$\mathfrak{H}$ of $\K$ which contains all curvature vector fields
(that is, such that for all $x\in M$ and all $v,w\in T_xM$,
$R^\gamma_x(v,w)\in \mathfrak{H}_x$), and satisfies
$\nabla^\gamma(\Sec(\mathfrak{H}))\subset\Sec(\mathfrak{H})$.  (That
there is a least Lie algebra sub-bundle with these properties is a
consequence of the following result, whose proof is to be found in
Appendix~2, Corollary~\ref{capLie}:\ let $E\to M$ be a Lie algebra
bundle equipped with a Lie connection, and let $E^1$ and $E^2$ be Lie
algebra sub-bundles of $E$ such that $\nabla(\Sec(E^i))\subset
\Sec(E^i)$, $i=1,2$; then $E^1\cap E^2$ is a Lie algebra bundle.)  The
holonomy algebra at $x$ is $\mathfrak{H}_x$.

We now prove a version of the Ambrose-Singer Theorem in the present
context.  For this purpose we define, for every $x\in M$, a vector
subspace $\mathfrak{R}_x$ of the vector space $\K_x$, as follows.
Since for any $y\in M$ and any curve $c$ in $M$ joining $y$ to $x$
parallel translation $\tau^\gamma_c$ along $c$ maps $\K_y$ to $\K_x$,
$\tau^\gamma_cR^\gamma_{y}(v,w)\in\K_x$ for any $v,w\in T_yM$.  We define
$\mathfrak{R}_x$ to be the least subspace of $\K_x$ containing all
parallel translates of curvature vector fields to $x$, that is, all
$\tau^\gamma_cR^\gamma_{y}(v,w)$.

\begin{thm}\label{AmSin}
$\mathfrak{H}_x=\mathfrak{R}_x$.
\end{thm}

In the proof we use the following result.  Let $\pi:E\to M$ be a
vector bundle equipped with a linear connection $\nabla$.  Let $E'$ be
a subset of $E$ such that:\ $\pi$ maps $E'$ onto $M$; for each $x\in M$,
$E'_x=\pi|_{E'}^{-1}(x)$ is a linear subspace of $E_x$; and for any
pair of points $x,y\in M$ and any curve $c$ in $M$ joining them,
$\tau_c(E'_x)\subseteq E'_y$.  Then $E'$ is a vector sub-bundle of $E$
and $\nabla(\Sec(E'))\subset\Sec(E')$.  The proof is to be found
in Appendix~2, Proposition~\ref{vbundle}.

\begin{proof}
Set $\mathfrak{R}=\bigcup_x\mathfrak{R}_x$.  It follows from
Proposition~\ref{vbundle} (with $E=\K$ and $E'=\mathfrak{R}$) that
$\mathfrak{R}$ is a vector sub-bundle of $\K$, and 
$\nabla^\gamma(\Sec(\mathfrak{R}))\subset \Sec(\mathfrak{R})$. 
Thus for any $V\in\Sec(\mathfrak{R})$ and any $X,Y\in\X(M)$,
\[
\nabla^\gamma_{X}\nabla^\gamma_{Y}V-\nabla^\gamma_{Y}\nabla^\gamma_{X}V
-\nabla^\gamma_{[X,Y]}V
=[V,R^\gamma(X,Y)]\in\Sec(\mathfrak{R}).
\]
Thus
$[\tau^\gamma_cR^\gamma_y(v_y,w_y),R^\gamma_x(v_x,w_x)]\in\mathfrak{R}_x$
for any $v_x,w_x\in T_xM$ and $v_y,w_y\in T_yM$.  But since
$\nabla^\gamma$ is a Lie connection, parallel translation preserves
brackets (Proposition~\ref{Lie}), so for any $y,z\in M$, any curve $c$
in $M$ joining $y$ to $x$ and any curve $d$ in $M$ joining $z$ to $x$,
$[\tau^\gamma_cR^\gamma_y(v_y,w_y),\tau^\gamma_dR^\gamma_{z}(v_z,w_z)]
\in\mathfrak{R}_x$.  Thus $\mathfrak{R}_x$ is a Lie subalgebra of
$\K_x$.  Moreover, the Lie algebras $\mathfrak{R}_x$ at different
points are isomorphic, so $\mathfrak{R}$ is a Lie algebra sub-bundle
of $\K$, and as we have pointed out already,
$\nabla^\gamma(\Sec(\mathfrak{R}))\subset\Sec(\mathfrak{R})$.
Evidently $R^\gamma_x(v,w)\in\mathfrak{R}_x$ for every $v,w\in T_xM$.
It follows that $\mathfrak{H}\subseteq\mathfrak{R}$.  On the other
hand, any Lie algebra sub-bundle $\K' $ of $\K$ which contains all
curvature vector fields and satisfies
$\nabla^\gamma(\Sec(\K'))\subset\Sec(\K')$ must contain all the
parallel translates of curvature vector fields, and so must contain
$\mathfrak{R}$.  Thus $\mathfrak{R}\subseteq\mathfrak{H}$, and so
$\mathfrak{R}=\mathfrak{H}$.  Thus $\mathfrak{R}_x=\mathfrak{H}_x$.
\end{proof}
 
\section{Holonomy of Landsberg spaces}

We now discuss the application of the results of the previous section
to Landsberg spaces in Finsler geometry.  (Some earlier and much more
rudimentary steps in this direction are to be found in \cite{MC}.)

It is well known \cite{Bao,BCS,Shen} that a Finsler space over a
manifold $M$ supports a canonical nonlinear connection or horizontal
distribution, that is, a distribution on the slit tangent bundle
$\pi:\TMO\to M$ which is everywhere transverse to the fibres.  (The
term `connection' has two distinct uses in this section --- a second
one will be introduced below --- not to mention the uses in the
context of the groupoid theory.  We therefore prefer the term
horizontal distribution since it reduces the possibilities of
confusion.)  We shall not need the explicit definition of the
horizontal distribution, only the fact that it exists, that it is
positively-homogeneous, and that it enjoys one or two other properties
to be introduced in due course.

Suppose given a Finsler space, together with its canonical horizontal
distribution.  A smooth curve in $\TMO$ is horizontal if its tangent
vector is everywhere horizontal; the definition extends to
piecewise-smooth curves in the obvious way.  Let $c$ be a
piecewise-smooth curve in $M$, $x$ a point in the image of $c$ and
$u\in\TxMO$.  A piecewise-smooth curve in $\TMO$ which is horizontal,
projects onto $c$ and passes through $(x,u)$ is called a horizontal
lift of $c$ through $u$.  If $c$ is defined on the finite closed
interval $[a,b]$ and $u\in\TxMO$, $x=c(a)$, then there is a unique
horizontal lift of $c$ through $u$ which is also defined on $[a,b]$.
If one makes a reparametrization of $c$, carrying out the same
reparametrization on its horizontal lift produces a horizontal lift of
the reparametrization of $c$.  As a consequence of this
reparametrization property we may standardize the domains of
definition of curves to $[0,1]$.  We may therefore make the definition
of the horizontal lift more specific, as follows.  Let $c:[0,1]\to M$
be any piecewise-smooth curve, with $c(0)=x$.  For any $u\in\TxMO$ the
horizontal lift of $c$ to $u$ is the unique piecewise-smooth
horizontal curve $\hlift{c}_u:[0,1]\to\TMO$ such that
$\pi\circ\hlift{c}_u=c$ and $\hlift{c}_u(0)=u$.

For any $x,y\in M$ and any piecewise-smooth curve $c$ with $c(0)=x$,
$c(1)=y$, we may define a map $\rho_c:\TxMO\to\TyMO$ by setting
$\rho_c(u)=\hlift{c}_u(1)$, which is in fact a diffeomorphism.  (The
map $\rho_c$ is sometimes called nonlinear parallel transport
\cite{Bao}, but again we prefer to avoid this term for the sake of
clarity.) Slightly more generally, for $t\in [0,1]$ we define
$\rho_c(t):\TxMO\to T_{c(t)}^\circ M$ by
$\rho_c(t)(u)=\hlift{c}_u(t)$.

The canonical horizontal distribution of a Finsler space has the
property that the Finsler function $F$ is constant along horizontal
curves.  Because of this, and because the horizontal distribution is
positively-homogeneous, we may restrict our attention to the
indicatrix bundle $\indic=\{(x,u)\in\TMO:F(x,u)=1\}$; this is a fibre
bundle over $M$ with compact fibres (each indicatrix is diffeomorphic
to a sphere).  In particular, each $\rho_c$ restricts to a
diffeomorphism of indicatrices, $\rho_c:\indic_x\to \indic_y$.

For any vector field $X$ on $M$ we denote by $\hlift{X}$ its
horizontal lift to $\TMO$ with respect to the canonical horizontal
distribution; it is the unique horizontal vector field such that
$\pi_*\hlift{X}=X$.  Another version of the constancy property is that
for every $X\in\X(M)$, $\hlift{X}(F)=0$; in particular, $\hlift{X}$ is
tangent to $\indic$.

In any Finsler space the fundamental tensor $g$ defines a fibre metric
on $\TMO$, and by restriction one on $\indic$.  A Landsberg space,
according to one possible definition (see for example \cite{Bao}), is
a Finsler space for which, for every pair of points $x,y\in M$ and
every curve $c$ joining $x$ and $y$, $\rho_c$ is an isometry of the
Riemannian spaces $(\indic_x,g_x)$ and $(\indic_y,g_y)$.  In the light
of the previous section we may put things another way:\ a Finsler
space is a Landsberg space if the holonomy groupoid of its canonical
horizontal lift, acting on its indicatrix bundle, is a subgroupoid of
the fibre-isometry groupoid.  Then the connection $\Gamma$
associated with the canonical horizontal lift is just
$c^\Gamma(t)=\rho_c(t)$.

There is an equivalent differential version of this definition.  A
Landsberg space is one for which $\Lie_{\hlift{X}}g=0$, or in other
words $\hlift{X}\in\J$, for all $X\in\X(M)$.  It is certainly the
case, therefore, that for a Landsberg space $\J$ is transitive, and we
conclude from Theorem~\ref{isogpoid} that the fibre-isometry groupoid
is a Lie groupoid.  We denote by $\I$ its Lie algebroid, as before.
It is moreover the case that the horizontal lift defines an
infinitesimal connection on $\I$:\ for $v\in T_xM$ we define
$\gamma(v)=\hlift{v}\in\I_x$ (where $\hlift{v}$ is considered as a
vector field along $\indic_x\to \indic$ satisfying the
isometry equation at $x$).

The theory of the previous section applies to Landsberg spaces, 
therefore. There is one point of interest which deserves closer 
inspection, and that concerns the identification of the covariant 
derivative operator $\nabla^\gamma$. To discuss this point it will be 
worthwhile to revert temporarily to the consideration of the 
canonical horizontal distribution as a distribution on $\TMO$.  

We may associate with every vector field $X$ on $M$ an operator
$\nabla_X$ on vertical vector fields $V$ on $\TMO$ by
\[
\nabla_XV=[\hlft{X},V].
\]
Then $\nabla_XV$ is vertical, is $\R$-linear in both arguments, and 
for $f\in C^\infty(M)$ 
\[
\nabla_{fX}V=f\nabla_XV,\qquad
\nabla_X(fV)=f\nabla_XV+X(f)V;
\]
so $\nabla$ has the properties of a (linear) covariant differentiation
operator.  Furthermore, $\nabla_X$ is a derivation of the bracket:
\[
\nabla_X[V,W]=[\nabla_XV,W]+[V,\nabla_XW]
\]
by the Jacobi identity.  Of course, to call this a covariant
derivative would be stretching a point, since to do so we should have
to regard the space of vertical vector fields on $\TxMO$ as the fibre
of a vector bundle over $M$, and a vertical vector field on $\TMO$ as
a section of this bundle.  

If the canonical horizontal distribution is spanned locally by vector
fields
\[
H_i=\fpd{}{x^i}-\Gamma^j_i\vf{u^j}
\]
then 
\[
\nabla_{\partial/\partial x^i}\left(\vf{u^j}\right)=
\left[H_i,\vf{u^j}\right]=\fpd{\Gamma^k_i}{u^j}\vf{u^k}.
\]
Now the connection coefficients $\conn{k}{i}{j}$ of the Berwald
connection are given by
\[
\fpd{\Gamma^k_i}{u^j}=\conn{k}{i}{j}.
\]  
That is to say, $\nabla$ is closely related to the Berwald connection.
The Berwald connection is usually thought of as a connection on the
vector bundle $\pi^*TM$, the pullback of $TM$ over $\TMO$; it is of
course a linear connection.  But for us it will be more convenient to
work with vertical vector fields on $\TMO$ than with sections of
$\pi^*TM$, though there is no essential difference (every section of
$\pi^*TM$ defines a vertical vector field on $\TMO$ by the vertical
lift procedure, and the correspondence is 1-1).  In fact for a curve
$c$ on $M$, $\nabla_{\dot{c}}$ is effectively covariant
differentiation with respect to the Berwald connection along
horizontal lifts of $c$.

There is a further point of interest, concerned with the differential
of the map $\rho_c$ for a curve $c$ joining $x$ and $y$ in $M$.
What we have to say initially applies to any Finsler space, and it
will continue to be convenient initially to consider the action of
$\rho_c$ on $\TMO$.  Since $\TxMO$ is a vector space (though deprived
of its origin), its tangent space at any point may be canonically
identified with itself (with origin restored); so for any $u\in\TxMO$,
$\rho_{c*u}$ may be regarded as a (linear) map $T_xM\to T_yM$.  This
map may be described as follows.  For $v\in T_xM$, $u\in\TxMO$, let
$v(t)$ be the (unique) vector field along $\hlift{c}_u$ which is
parallel with respect to the Berwald connection of the Finsler space
and satisfies $v(0)=v$:\ then $\rho_{c*u}(v)=v(1)$.  This may be seen
as follows.  Consider the pullback $c^*\TMO$.  The canonical
horizontal distribution on $\TMO$ induces one on $c^*\TMO$, which is
1-dimensional and is spanned by the vector field
\[
\hlift{\left(\vf{t}\right)}=\vf{t}-\dot{c}^j\Gamma^i_j(c^k,u^k)\vf{u^i}.
\]
In view of the definition of $\rho_c$ it is clear that $\rho_{c*u}(v)$
is the Lie translate by $\hlift{(\partial/\partial t)}$ of $v$,
considered as a vertical vector at $u$, to $\rho_c(u)$.  For any
vertical vector field
\[
V=V^i(t,u^k)\vf{u^i}
\]
on $c^*\TMO$,
\begin{align*}
\left[\hlift{\left(\vf{t}\right)},V^i\vf{u^i}\right]&=
\left(\hlift{\left(\vf{t}\right)}(V^i)+\conn{i}{j}{k}\dot{c}^jV^k\right)\vf{u^i}\\
&=\left(\hlift{\dot{c}}_u(V^i)+\conn{i}{j}{k}\dot{c}^jV^k\right)\vf{u^i}.
\end{align*}
Thus $V$ is Lie transported by $\hlift{(\partial/\partial t)}$ if and
only if 
\[
\hlift{\dot{c}}_u(V^i)+\conn{i}{j}{k}\dot{c}^jV^k=0,
\]
that is, if and only if it is parallelly translated along
$\hlift{c}_u$ with respect to the Berwald connection.  We can
conveniently summarise this in terms of $\nabla$, as follows.  For any
curve $c$ on $M$ with $c(0)=x$, $c(1)=y$, the map $\rho_{c*}$,
considered as a map from vertical vector fields on $\TxMO$ to vertical
vector fields on $\TyMO$, is given by parallel translation with
respect to $\nabla$, in the sense that if $V(t)$ is a field of
vertical vector fields along $c(t)$ with $\nabla_{\dot{c}}V=0$, then
$\rho_{c*}V(0)=V(1)$.  

We can formally compute the curvature of $\nabla$; we obtain
\[
\nabla_X\nabla_YV-\nabla_Y\nabla_XV-\nabla_{[X,Y]}V
=[V,R(X,Y)]
\]
where $R$ is the curvature of the canonical horizontal distribution;
note that $R(X,Y)$ is a vertical vector field. 

One version of the Landsberg property is that $\Lie_{\hlift{X}}g=0$ 
for all $X\in\X(M)$. But
\begin{align*}
\Lie_{\hlift{X}}g(V,W)&=
\hlift{X}\big(g(V,W)\big)-g\big([\hlift{X},V],W\big)-g\big(V,[\hlift{X},W]\big)\\
&=\hlift{X}\big(g(V,W)\big)-g\big(\nabla_XV,W\big)-g\big(V,\nabla_XW\big); 
\end{align*}
so if we extend the action of $\nabla$ to fibre metrics in the 
obvious way we may write the condition as $\nabla g=0$. 

We return to consideration of $\indic$, the Lie algebroid $\I$ and its
kernel $\K$, in a Landsberg space.  We now consider $\hlift{X}$, for
any $X\in\X(M)$, as a vector field on $\indic$; then in a Landsberg
space we have $\hlift{X}\in\Sec(\I)$.  The infinitesimal connection is
just $\gamma(X)=\hlift{X}$, and $\nabla^\gamma$ is the restriction of
the operator $\nabla$ to $\Sec(\K)$.  Since $\K$ is a vector bundle
the restriction of $\nabla$ is a genuine covariant derivative. From 
the theory of the previous section we deduce the following proposition 
(which may also be proved directly from the definitions of $\K$ and 
$\nabla$).

\begin{prop}
In a Landsberg space 
\begin{enumerate}
\item $\nabla(\Sec(\K))\subset\Sec(\K)$;
\item for all $X,Y\in\X(M)$, $R(X,Y)\in\Sec(\K)$.
\end{enumerate}
\end{prop}

We can now apply the results of the previous section to the case of a
Landsberg space.  We have a locally trivial Lie groupoid, the fibre-isometry
groupoid of $\indic$, whose Lie algebroid $\I$ is equipped with an
infinitesimal connection $\gamma:TM\to \I$, where for $v\in T_xM$,
$\gamma(v)=\hlift{v}$.  The corresponding connection $\Gamma$ is the
map $c\mapsto\rho_c$ (where $\rho_c$ is regarded as a curve in the
fibre-isometry groupoid).  The covariant derivative operator
$\nabla^\gamma$ corresponding to $\gamma$, which acts on sections of
the kernel, is just $\nabla$ operating on $\Sec(\K)$.  The holonomy
groupoid is a Lie groupoid, and the holonomy Lie algebra bundle is the
least Lie algebra sub-bundle $\mathfrak{H}$ of $\K$ which contains all
curvature vector fields $R_x(v,w)$ for $x\in M$ and $v,w\in T_xM$ and
satisfies $\nabla(\Sec(\mathfrak{H}))\subset\Sec(\mathfrak{H})$.  We
may further specialise the Ambrose-Singer Theorem,
Theorem~\ref{AmSin}, from the previous section.  For this purpose we
recall that the parallel translation operator corresponding to
$\nabla$ along a curve $c$ from $y$ to $x$ is the restriction to
$\K_y$ of $\rho_{c*}$, the differential of $\rho_c$.  This is an
isomorphism of Lie algebras $\K_y\to\K_x$, as follows from the general
theory, and also directly from the fact that $\rho_c$ is an isometry.
We now define $\mathfrak{R}_x$ to be the least subspace of $\K_x$
containing all elements of the form $\rho_{c*}R_{y}(v,w)$.

\begin{thm}
In a Landsberg space $\mathfrak{H}_x=\mathfrak{R}_x$.
\end{thm}

\section{Covariant derivatives of the curvature}

This section is devoted to clarifying the relationship between the
holonomy algebra $\mathfrak{H}_x$ of a Landsberg space and the algebra
generated by the curvature and its covariant derivatives.  It is based
on Section~10 of Chapter~II of Kobayashi and Nomizu~\cite{KandN}.

We define, for each $k=2,3,\ldots$\,, an $\R$-linear space of vertical
vector fields $\C^k$ on $\indic$ inductively as follows:\ $\C^2$ is
the set of all finite linear combinations over $\R$ of vertical vector
fields of the form $R(X,Y)$ for any pair of vector fields $X$, $Y$ on
$M$; $\C^k$ is the set of all finite linear combinations over $\R$ of
vertical vector fields which either belong to $\C^{k-1}$ or are of the
form $\nabla_XV$ for some vector field $X$ on $M$, where
$V\in\C^{k-1}$.

\begin{thm} With $\C^k$ defined as above
\begin{enumerate}
\item each $\C^k$ is a $C^\infty(M)$-module; 
\item $[\C^k,\C^l]\subset\C^{k+l}$;
\item for $m=0,1,2,\ldots$\,, $\C^{m+2}/\C^{m+1}$ is spanned by equivalence 
classes of vertical vector fields of the form
\[
\nabla_{X_1}\nabla_{X_2}\cdots\nabla_{X_m}\left(R(Y,Z)\right),
\]
where $X_1$, $X_2$, \ldots\,, $X_m$, $Y$ and $Z$ are any vector fields
on $M$ (and $\C^1=\{0\}$). 
\end{enumerate}
\end{thm}

(Note that in the third assertion, every $\nabla$ operates on a
vertical vector field.  There is no question of writing this formula
in terms of covariant differentials of the curvature tensor, as one
would for a linear connection, essentially because (for example)
$\nabla_{X_a}Y$ would make no sense in this context.)

\begin{proof}
(1) We have to show that $\C^k$ is closed under scalar multiplication
by smooth functions on $M$. For any $f\in C^\infty(M)$,
$fR(X,Y)=R(fX,Y)\in\C^2$.  Suppose that $\C^{k-1}$ is closed under
scalar multiplication by smooth functions on $M$.  Then for any
$V\in\C^{k-1}$ we know that $fV\in\C^{k-1}\subset\C^k$, and
$f\nabla_XV=\nabla_{fX}V\in\C^k$; so $\C^k$ is closed under scalar
multiplication by smooth functions on $M$.

(2) We show first that $[\C^k,\C^2]\subset\C^{k+2}$ for all $k\geq 2$.  We
have
\[
[V,R(X,Y)]=
\nabla_X\nabla_YV-\nabla_Y\nabla_XV-\nabla_{[X,Y]}V,
\]
and if $V\in\C^k$, each term on the right-hand side belongs to
$\C^{k+2}$.  Now suppose that $[\C^k,\C^l]\subset\C^{k+l}$ for all $k$
and all $l<L$.  Let $V\in\C^L$.  It will be sufficient to consider
$V=\nabla_XW$ with $W\in\C^{L-1}$.  Then
\[
[U,V]=[U,\nabla_XW]=\nabla_X[U,W]-[\nabla_XU,W].
\]
For $U\in\C^k$, by the induction assumption
$[U,W]\in\C^{k+L-1}$, so $\nabla_X[U,W]\in\C^{k+L}$;
while $\nabla_XU\in\C^{k+1}$, so by the induction assumption
again $[\nabla_XU,W]\in\C^{k+L}$.

(3)\quad Obvious.
\end{proof}

Let us set $\bigcup_k\C^k=\C$.  Then $\C$ is a
$C^\infty(M)$-linear space of vertical vector fields on $\indic$ contained
in $\Sec(\K)$, which is closed under bracket, contains $R(X,Y)$ and is
invariant under $\nabla$; and it is minimal with respect to these
properties.  It is clear, moreover, that 
$\C^k\subset\Sec(\mathfrak{H})$
for all $k=2,3,\ldots$\, so $\C\subset\Sec(\mathfrak{H})$.
However, we have no guarantee that $\C$ consists of the sections
of a Lie algebra bundle, so in particular we cannot identify it with
$\Sec(\mathfrak{H})$.

For $x\in M$ we set $\C_x=\{V_x:V\in\C\}$.  Then $\C_x$ consists of
vertical vector fields on $\indic_x$, and is a Lie algebra with respect to
the bracket of vertical vector fields.  Clearly, $\C_x$ is a Lie
subalgebra of the holonomy algebra $\mathfrak{H}_x$ for each $x\in M$.
However, we have no reason to suppose that $\C_x$ and $\C_y$ are
isomorphic Lie algebras for $x\neq y$, or even that they have the same
dimension.

\begin{prop}\label{dim}
For each $x\in M$ there is a neighbourhood $U$ of $x$ in $M$ such 
that for all $y\in U$, $\dim\C_y\geq\dim\C_x$.
\end{prop}

\begin{proof}
Every element of $\C_x$ is the restriction to $\indic_x$ of an
element of $\C$.  Let $V_a$,
$a=1,2,\ldots,\dim\C_x$, be elements of $\C$
such that $\{V_a|_x\}$ is a basis for $\C_x$.  Since the
$V_a|_x$ are linearly independent, there is a neighbourhood $U$ of $x$
such that the $V_a|_y$ are linearly independent for all $y\in U$.
Then $\dim\C_y\geq\dim\C_x$ for all $y\in U$.
\end{proof}

\begin{prop}
Let $\mu$ be the maximum value of $\dim\C_x$ as $x$ ranges over 
$M$. Then the set $\{x\in M:\dim\C_x=\mu\}$ is an open subset of 
$M$. Let $U$ be a path-connected component of it:\ then for each 
$x\in U$, $\C_x$ is the holonomy algebra at $x$ of the 
Landsberg space defined over $U$ by the restriction of the Landsberg 
Finsler function $F$ to $T^\circ U$.
\end{prop}

\begin{proof}
By Proposition \ref{dim}, every point $x$ with $\dim\C_x=\mu$ has a
neighbourhood at all of whose points $y$, $\dim\C_y=\mu$; so $\{x\in
M:\dim\C_x=\mu\}$ is open.  Let $U$ be a path-connected component
of this set, and restrict both the Landsberg structure and $\C$ to
this open submanifold of $M$.  If $V_a$, $a=1,2,\ldots,\mu$, are
elements of $\C|_U$ such that $\{V_a|_x\}$ is a basis for $\C_x$ for
some $x\in U$, then they also form a basis for $\C_y$ for all $y$ in
some neighbourhood of $x$ in $U$; so $\bigcup_{x\in U}\C_x$ is a
vector bundle over $U$, and $\C|_U$ consists of its sections.  Each of
its fibres is a Lie algebra, and $\nabla$ is a Lie connection for the
corresponding bracket of sections:\ so it is a Lie algebra bundle.  We
denote it by $\overline{\K}$.  Now $\overline{\K}$ evidently contains
all curvature vector fields $R_x(v,w)$ for $x\in U$, $v,w\in T_xU$,
and of course $\nabla(\Sec(\overline{\K}))\subset\Sec(\overline{\K})$.
Moreover, by construction $\overline{\K}$ is contained in any Lie
algebra sub-bundle of $\K|_U$ with these properties.  It is therefore
the least such Lie algebra bundle, which is to say that it is the
holonomy Lie algebra bundle of the restriction of the Landsberg
structure to $U$.
\end{proof}

\begin{cor}
If $\dim\C_x$ is constant on $M$ then $\overline{\K}=\mathfrak{H}$, and
$\C_x=\mathfrak{H}_x$ for all $x\in M$.
\end{cor}

\begin{thm}
If the data are analytic (that is, if $M$ is an analytic manifold and 
$F$ an analytic function) then $\C_x=\mathfrak{H}_x$ 
for all $x\in M$.
\end{thm}

\begin{proof}
We shall show that $\dim\C_x$ is constant on $M$, so the result
will follow from the corollary above.  We proceed via a couple of
lemmas.

\begin{lem}\label{exp}
Take any $x\in M$ and any analytic curve $c(t)$ with $c(0)=x$.  Let
$V$ be any analytic section of $\mathfrak{H}$. For $|t|$ sufficiently small
\[
V_{c(t)}=\rho_{c*}|_0^t\left(
\sum_{r=0}^\infty\frac{t^r}{r!}(\nabla_{\dot{c}}^rV)_x\right)
\]
where $\rho_{c*}|_0^t$ is the operator of parallel translation along 
$c$ from $c(0)=x$ to $c(t)$.
\end{lem}

\begin{proof}
Let $\{v_\alpha(0)\}$ be a basis for $\mathfrak{H}_x$, and for each
$\alpha$ let $v_\alpha(t)=\rho_{c*}|_0^tv_\alpha(0)$ be the element of
$\mathfrak{H}_{c(t)}$ defined by parallel translation of $v_\alpha(0)$.
Then $\{v_\alpha\}$ is a basis of sections of $\mathfrak{H}$ along $c$,
and $\nabla_{\dot{c}}v_\alpha=0$ by construction.  An analytic section
$V$ of $\mathfrak{H}$, restricted to $c$, may be expressed in terms of
$\{v_\alpha\}$ as
\[
V_{c(t)}=\nu^\alpha(t)v_\alpha(t),
\]
where the coefficients $\nu^\alpha$ are analytic functions of $t$.
Thus
\[
\nu^\alpha(t)=\sum_{r=0}^\infty\frac{t^r}{r!}\frac{d^r\nu^\alpha}{dt^r}(0)
\]
for $|t|$ sufficiently small:\ to be precise, for $|t|$ less
than the smallest of the radii of convergence of the functions
$\nu^\alpha$ at $0$.  For each $r=1,2,\ldots$\,,
\[
(\nabla_{\dot{c}}^rV)(t)=\frac{d^r\nu^\alpha}{dt^r}(t)v_\alpha(t)
\]
since $\nabla_{\dot{c}}v_\alpha=0$. Thus  
\[
V_{c(t)}=\left(
\sum_{r=0}^\infty\frac{t^r}{r!}\frac{d^r\nu^\alpha}{dt^r}(0)\right)
\rho_{c*}|_0^tv_\alpha(0)
=\rho_{c*}|_0^t\left(
\sum_{r=0}^\infty\frac{t^r}{r!}(\nabla_{\dot{c}}^rV)_x\right),
\]
as required.
\end{proof}

We wish to apply this formula with $V$ an element of
$\C\subset\Sec(\mathfrak{H})$.  Now $\C$ is spanned (over
$C^\omega(M)$) by elements of the form
\[
\nabla_{X_1}\nabla_{X_2}\cdots\nabla_{X_m}\left(R(Y,Z)\right),
\]
where $X_1$, $X_2$, \ldots\,, $X_m$, $Y$ and $Z$ are any vector fields
on $M$.  If we restrict our attention to a coordinate patch
$U$ from the given analytic atlas then it will be enough to take 
$X_1$, $X_2$, \ldots\,, $X_m$, $Y$
and $Z$ to be coordinate fields.  We take a local basis $\{V_\alpha\}$
of sections of $\mathfrak{H}$ over $U$ and set
\[
R\left(\vf{x^j},\vf{x^k}\right)=R_{jk}=R^\alpha_{jk}V_\alpha;
\]
the coefficients $R^\alpha_{jk}$ are functions on $U$, analytic in the
coordinates.  We abbreviate $\nabla_{\partial/\partial x^i}$ to
$\nabla_i$; since 
$\nabla(\Sec(\mathfrak{H}))\subset\Sec(\mathfrak{H})$ we may set
\[
\nabla_iV_\alpha=\conn{\beta}{i}{\alpha}V_\beta;
\]
the connection coefficients $\conn{\beta}{i}{\alpha}$ are functions on
$U$, analytic in the coordinates.  Thus for any $x\in U$ there is an
open Euclidean coordinate ball $B_\delta(x)$ centred at $x$ such that
the power series expansions about $x$, in terms of coordinates, of all
the functions $R^\alpha_{jk}$ and $\conn{\beta}{i}{\alpha}$ converge
in $B_\delta(x)$.

\begin{lem}\label{conv}  
The coefficients with respect to $\{V_\alpha\}$ of all sections of
$\mathfrak{H}$ of the form
$\nabla_{i_1}\nabla_{i_2}\cdots\nabla_{i_m}R_{jk}$ can be expanded in
convergent power series about $x\in U$ in {\em the same\/} open
coordinate ball $B_\delta(x)$.
\end{lem}

\begin{proof}
If $f$ is an analytic function whose power series expansion about $x$
converges in $B_\delta(x)$, the power series expansion about $x$ of
$\partial f/\partial x^i$ also converges in $B_\delta(x)$; and if
$f_1$ and $f_2$ are analytic functions whose power series expansions
about $x$ converge in $B_\delta(x)$, the power series expansion about
$x$ of $f_1f_2$ also converges in $B_\delta(x)$.  But
\[
\nabla_i(\nu^\alpha V_\alpha)=
\left(\fpd{\nu^\alpha}{x^i}+\conn{\alpha}{i}{\beta}\nu^\beta\right)V_\alpha.
\]
Thus if the coefficients with respect to $\{V_\alpha\}$ of an analytic
section $V$ of $\mathfrak{H}$ have power series expansions about $x$
that converge in $B_\delta(x)$ then the coefficients of $\nabla_iV$
have power series expansions about $x$ that also converge in
$B_\delta(x)$.  It follows that the coefficients of all sections of
$\mathfrak{H}$ of the form
$\nabla_{i_1}\nabla_{i_2}\cdots\nabla_{i_m}R_{jk}$ can be expanded in
convergent power series about $x$ in $B_\delta(x)$.
\end{proof}

We can now complete the proof of the theorem.

Let $U$ be a coordinate neighbourhood of $x$ in $M$ (from the analytic atlas)
with $x$ as origin, whose image in $\R^m$ contains the ball $B_\delta(x)$.  For
any fixed $w\in T_xM$ whose coordinate representation has Euclidean
length $\delta$ let $r_w$ be the ray starting at $x$ with initial
tangent vector $w$, so that in coordinates $r_w(t)=(tw^i)$ with
$\sum(w^i)^2=\delta^2$.  Set $W=w^i\partial/\partial x^i$, so that
$r_w$ is the integral curve of $W$ starting at $x$.  By
Lemma~\ref{conv} the formula of Lemma~\ref{exp}, which here becomes
\[
V_{r_w(t)}=\rho_{r_w*}|_0^t\left(
\sum_{r=0}^\infty\frac{t^r}{r!}(\nabla_{W}^rV)_x\right),
\]
holds for all $t$ with $0\leq t<1$, and for all $w$, where $V$ is {\em
any\/} vector field of the form
$\nabla_{i_1}\nabla_{i_2}\cdots\nabla_{i_m}R_{jk}$.  The $V_{r_w(t)}$
span $\C_{r_w(t)}$; each term in the sum belongs to
$\C_x$.  That is to say, for each $t$ with $0\leq t<1$,
$\C_{r_w(t)}\subseteq\rho_{r_w*}|_0^t(\C_x)$.  Since
parallel translation is an isomorphism,
$\dim\C_{r_w(t)}\leq\dim\C_x$.  There is thus a
neighbourhood of $x$ in $M$ (namely $B_\delta(x)$) on which
$\dim\C_y\leq\dim\C_x$.  But we know that there is also a
neighbourhood of $x$ in $M$ on which
$\dim\C_y\geq\dim\C_x$; so there is a neighbourhood of $x$
on which $\dim\C_y=\dim\C_x$.  Thus $\dim\C_x$ is
locally constant on $M$, and so (assuming $M$ is connected) it is
constant on $M$.
\end{proof}

\section{Concluding remarks}

In this paper we have used techniques from the theory of Lie groupoids
and Lie algebroids to investigate the infinitesimal holonomy of a
certain class of fibre bundles whose the fibres are compact and
support a fibre metric.  We chose this particular case because it has
an immediate application to the holonomy of the canonical horizontal
lift or nonlinear connection of a Landsberg space in Finsler geometry,
where the horizontal lifts of curves determine isometries between the
fibres.  We should, of course, like to generalise this approach, so as
to be able to deal with Finsler spaces in general, and more generally
with homogeneous nonlinear connections as in \cite{Kozma}.

To do so we have to face a difficulty.  One important role of the
fibre metric in this paper was to ensure that the holonomy groupoid,
in the case of a Landsberg space for example, is a subgroupoid of a
Lie groupoid, namely the fibre-isometry groupoid.  Without it, we have
no obvious choice of an ambient Lie groupoid within which to locate
the holonomy groupoid, even for Finsler spaces.  Indeed, Muzsnay and
Nagy \cite{MuzNag} give an example of a Finsler space whose holonomy
group is not a (finite-dimensional) Lie group.  However, we can always
work within the groupoid of fibre diffeomorphisms, though it will not
have a finite-dimensional smooth structure.  We expect, however, that
it will still be possible to construct the equivalent of a Lie
algebroid for this more general structure, and to represent sections
of this `algebroid' as projectable vector fields on the total space of
the bundle.  This expectation is based on material to be found in
\cite{Michor}; the fact that the fibres are actually (for Finsler 
spaces) or effectively (for homogenous nonlinear connections) compact 
is significant.

If this is indeed correct, then it seems to us that what remains of
our account above when references to fibre metrics and fibre
isometries are omitted should be a good guide to a theory of holonomy
of Finsler spaces and homogeneous nonlinear connections.  In
particular, one can think of the horizontal lift as providing an
`infinitesimal connection' $\gamma$, with respect to which each
projectable vector field can be written uniquely as the sum of a
horizontal and a vertical vector feld.  The vertical vector fields
should then be thought of as sections of the `kernel' of the
`algebroid'; that is, we now take seriously the suggestion made in
Section~4 that the space of vertical vector fields on $\TxMO$ should
be regarded as the fibre of a `vector bundle', indeed `Lie algebra
bundle', $\V$ over $M$.  The covariant-derivative-like operator
$\nabla$, which acts on $\Sec(\V)$, and is related to the
Berwald connection in the Finsler case or more generally to a
so-called connection of Berwald type \cite{CBer}, will continue to
play the role of $\nabla^\gamma$ for the `infinitesimal connection'.
The objects of interest are the `Lie algebra sub-bundles'
$\mathfrak{L}$ of $\V$ which contain all curvature vector fields (that
is, such that for all $x\in M$ and all $v,w\in T_xM$,
$R_x(v,w)\in \mathfrak{L}_x$), and satisfy
$\nabla(\Sec(\mathfrak{L}))\subset\Sec(\mathfrak{L})$.

Let us for definiteness consider the Finsler case.  For any $x\in M$
we may define the vector subspace $\mathfrak{R}_x$ of $\V_x$, the
space of vertical vector fields on $\indic_x$, much as before:\ it is
the smallest subspace of $\V_x$ containing all fields
$\rho_{c*}R_{y}(v,w)$ for any $y\in M$ and any curve $c$ in $M$
joining $y$ to $x$, and (now an extra condition) which is closed with
respect to the compact $C^\infty$ topology (for details see
\cite{Michor}).  Evidently for any curve $c$ joining $x$ and $y$,
$\rho_{c*}:\mathfrak{R}_x\to\mathfrak{R}_y$ is an isomorphism.  Set
$\mathfrak{R}=\bigcup_x\mathfrak{R}_x$.  Then by the ray construction
used in the proof of Proposition~\ref{isotriv}, $\mathfrak{R}$ is
locally trivial, in the sense that $\mathfrak{R}|_U\simeq
U\times\mathfrak{R}_o$ where $U$ is a suitable coordinate chart and
$o$ its origin.  We shall allow ourselves to think of $\mathfrak{R}$
as a vector bundle over $M$.  We claim that
$\nabla(\Sec(\mathfrak{R}))\subset \Sec(\mathfrak{R})$.  Indeed,
let $c$ be a curve in $M$ and $V$ a local section of $\mathfrak{R}$
over some $U$ containing $c(0)=x$, so that in particular
$V(t)=V_{c(t)}$ depends smoothly on $t$ and
$V(t)\in\mathfrak{R}_{c(t)}$.  Now
\[
\nabla_{\dot{c}(0)}V=
\frac{d}{dt}\left(\left(\rho_{c*}|_0^t\right)^{-1}V(t)\right)_{t=0},
\]
where $\rho_{c*}|_0^t:\mathfrak{R}_x\to\mathfrak{R}_{c(t)}$ is the
isomorphism.  Thus
$\left(\rho_{c*}|_0^t\right)^{-1}V(t)\in\mathfrak{R}_x$ for all $t$,
and so $\nabla_{\dot{c}(0)}V\in\mathfrak{R}_x$.  That is, for any
local section $V$ of $\mathfrak{R}$, for any $x$ in its domain and for
any $v\in T_xM$, $\nabla_vV\in\mathfrak{R}_x$.  We can now argue
exactly as in the proof of Theorem~\ref{AmSin} to conclude that
$\mathfrak{R}$ is a Lie algebra sub-bundle of $\V$.

The drawback, of course, is that one cannot use the theory of
Section~2 (which applies only to Lie groupoids) to argue that the
holonomy Lie algebra bundle is the least Lie algebra sub-bundle of
$\V$ which contains curvature fields and whose section space is
invariant under $\nabla$.  But it is certainly the case that
$\mathfrak{R}$ satisfies these latter requirements, and
$\mathfrak{R}_x$ is the obvious candidate for the role of holonomy
algebra at $x$.

In \cite{Michor}, Michor also proposes that (in our notation)
$\mathfrak{R}_x$ should be considered as the holonomy algebra at $x$;
his analysis is different from ours, but leads nevertheless to the
result that $\mathfrak{R}_x$ is a Lie algebra (his Lemma~12.3).

In a recent paper about holonomy of Finsler spaces \cite{MuzNag},
Muzsnay and Nagy express some reservations about this proposal of
Michor's:\ they say, not unreasonably on the face of it, that `the
introduced holonomy algebras [i.e.~the algebras $\mathfrak{R}_x$]
could not be used to estimate the dimension of the holonomy group
since their tangential properties to the holonomy group were not
clarified'.  It is true that Michor doesn't directly address the
question of how one gets from the holonomy group at $x$ to
$\mathfrak{R}_x$; his justification for calling $\mathfrak{R}_x$ the
holonomy algebra at $x$ seems to be that when $\mathfrak{R}_x$ is
finite-dimensional and generated by complete vector fields (as will be
automatically the case when the fibre is compact) then it really is
the holonomy algebra of a principal bundle with connection, as he
shows.  (We, of course, could argue in a somewhat similar vein, namely
by pointing out that on specializing for example to the Landsberg case
we obtain the correct answer.)  But in fact Muzsnay and Nagy have the
solution to their problem ready to hand, since a major part of
Section~2 of their paper is devoted to discussing what it means for a
vector field on $\indic_x$ to be tangent to the holonomy group, and
they show in Section 3 of their paper that the curvature fields
$R_x(v,w)$ are tangent to the holonomy group at $x$ according to their
definition. We shall show below, using their methods for the most part, 
that the elements of $\mathfrak{R}_x$ are indeed tangent to the 
holonomy group at $x$.

Before discussing this further, we wish to resolve a potential cause
of confusion between their paper and ours.  In \cite{MuzNag} Muzsnay
and Nagy introduce, for any Finsler space, the notion of the curvature
algebra at $x\in M$, which they denote by $\mathfrak{R}_x$:\ it is the
Lie subalgebra of $\V_x$ `generated by the curvature vector fields',
i.e.\ the vector fields $R_x(v,w)$ as $v$, $w$ range over $T_xM$, for
{\em fixed\/} $x$.  We point out, for the sake of clarity, that their
curvature algebra at $x$ is not the same as our $\mathfrak{R}_x$, and
in general will be a proper subalgebra of it.

Now the restricted holonomy group at $x$, $H^\circ_x$, which
corresponds to closed curves in $M$ which are contractible, is a
subgroup of the diffeomorphism group $\diff(\indic_x)$ of $\indic_x$.
Moreover, it has the property that every one of its elements can be
connected to the identity diffeomorphism $\id_{\indic_x}$ by a
piecewise-smooth curve lying in $H^\circ_x$; just as in Section~2 this
follows from the Factorization Lemma.

Let $\phi(t)$ be a curve in $\diff(\indic_x)$, defined on some open
interval $I$ containing $0$, such that $\phi(0)=\id_{\indic_x}$, and
which is differentiable enough for the following construction to make
sense.  For $u\in \indic_x$ consider the curve $\phi_u$ in $\indic_x$
defined by $\phi_u(t)=\phi(t)u$.  Set
\[
V_u=\frac{d^k}{dt^k}(\phi_u)_{t=0},
\]
where $k\geq 1$ is the smallest integer for which the derivative is
not zero for all $u$. Assume this defines a smooth vector field $V$ on $\indic_x$. 
Muzsnay and Nagy show in effect that if we take curves $\phi(t)$ in the
restricted holonomy group $H^\circ_x$ then the set of all such vector
fields is an $\R$-linear space, closed under bracket; that is to say,
it is a Lie algebra of vector fields.  (It is necessary to allow $k>1$
in order to deal with, for example, the problem with the
parametrization when defining the bracket in terms of the commutator
of curves in the group.)  But as we remarked above, every element of
$H^\circ_x$ is connected to the identity by a curve in $H^\circ_x$,
which makes this a pretty strong candidate to be called the holonomy
algebra; we denote it by $\hol(x)$ as in \cite{MuzNag}.

For any curve $c$ in $M$ joining $y$ to $x$, the map
$H^\circ_y\to H^\circ_x$ defined by $\psi\mapsto \rho_c\circ
\psi\circ\rho_c^{-1}$, where $\psi\in H^\circ_y$, is an isomorphism.
Now let $\psi(t)$ be a curve in $H^\circ_y$ with
$\psi(0)=\id_{\indic_y}$; then $\phi:t\mapsto \rho_c\circ
\psi(t)\circ\rho_c^{-1}$ is a curve in $H^\circ_x$ with
$\phi(0)=\id_{\indic_x}$.  For $u\in \indic_x$ we have
\[
\phi_u(t)=\rho_c\left(\psi_{\rho_c^{-1}(u)}(t)\right),
\]
so the element $V$ of $\hol(x)$ determined by $\phi$ is just
$\rho_{c*}W$ where $W$ is the element of $\hol(y)$ determined by
$\psi$.  It follows that $W\mapsto\rho_{c*}W$ is an isomorphism of Lie
algebras $\hol(y)\to\hol(v)$.  Now as Muzsnay and Nagy themselves
show, for every $y\in M$, $\hol(y)$ contains the curvature fields
$R_y(v,w)$.  Thus $\hol(x)$ must contain their parallel translates
$\rho_{c*}R_y(v,w)$.  It follows immediately that
$\mathfrak{R}_x\subset\hol(x)$.

\section*{Appendix 1:\ groupoids and algebroids}

In this section we recall the relevant definitions from the theory of
Lie groupoids and Lie algebroids (more details may be found in
Mackenzie's comprehensive work~\cite{MK05}).

\begin{defn}
A \emph{groupoid} consists of two sets $\Omega$ and $M$, called
respectively the \emph{groupoid} and the \emph{base}, together with
two maps $\alpha$ and $\beta$ from $\Omega$ to $M$, called
respectively the \emph{source projection} and \emph{target
projection}, a map $1 : x \mapsto 1_x, \; M \to \Omega$ called the
\emph{object inclusion map}, and a partial multiplication $(h,g)
\mapsto hg$ in $\Omega$ defined on the set $\Omega*\Omega = \{(h,g)
\in \Omega \times \Omega \; | \; \alpha(h) = \beta(g) \}$, all subject
to the following conditions:
\begin{enumerate}
\item $\alpha(hg) = \alpha(g)$ and $\beta(hg) = \beta(h)$ for all
$(h,g) \in \Omega*\Omega$; \item $j(hg) = (jh)g$ for all $j, h, g \in
\Omega$ such that $\alpha(j) = \beta(h)$ and $\alpha(h) = \beta(g)$;
\item $\alpha(1_x) = \beta(1_x) = x$ for all $x \in M$; \item
$g1_{\alpha(g)} = g$ and $1_{\beta(g)}g = g$ for all $g \in \Omega$;
\item each $g \in \Omega$ has a two-sided inverse $g^{-1}$ such that
$\alpha(g^{-1}) = \beta(g)$, $\beta(g^{-1}) = \alpha(g)$ and $g^{-1}g
= 1_{\alpha(g)}$, $gg^{-1} = 1_{\beta(g)}$.
\end{enumerate}
We use the notation $r_g$, $l_g$ to denote right and left translation
by $g \in \Omega$.

A \emph{Lie groupoid} is a groupoid $\Omega$ on base $M$ together with
smooth structures on $\Omega$ and $M$ such that the maps $\alpha,
\beta : \Omega \to M$ are surjective submersions, the object inclusion
map $x \mapsto 1_x, \; M \to \Omega$ is smooth, and partial
multiplication $\Omega*\Omega \to \Omega$ is smooth.  We say that
$\Omega$ is \emph{locally trivial} if the map $(\alpha, \beta) :
\Omega \to M \times M$ is a surjective submersion.  
\end{defn}

We adopt the convention that the elements of $\Omega*\Omega$ satisfy
$\alpha(h) = \beta(g)$ as this is appropriate when the elements of the
groupoid are maps whose arguments are on the right.

\begin{defn}
A \emph{Lie algebroid} is a vector bundle $\pi : A \to M$ together
with a Lie bracket on the (infinite-dimensional) vector space
$\Sec(\pi)$ of sections of $\pi$ and a vector bundle map $a : A \to
TM$ over the identity on $M$, the \emph{anchor map}, which satisfies
\[
[\theta, f\phi] = f[\theta, \phi] + (\Lie_{a \circ \theta}f) \phi
\]
for any sections $\theta, \phi \in \Sec(\pi)$ and any function $f$
on $M$, and whose action on sections by composition is a Lie algebra
homomorphism from $\Sec(\pi)$ to $\X(M)$, the Lie algebra of vector
fields on $M$.

We say that $A$ is \emph{transitive} if the anchor map is surjective
on each fibre.  
\end{defn}

The kernel $L$ of the anchor map $a$ is again a Lie algebroid, but of
a rather special kind:\ \emph{its} anchor map is by definition zero,
and so the Lie bracket on sections of $L \to M$ induces a Lie bracket
on each fibre of $L$.  It may be shown that if $A$ is transitive then
the fibres of $L$ are pairwise isomorphic as Lie algebras, so that $L$
is a Lie algebra bundle.

Every Lie groupoid $\Omega$ gives rise to an associated Lie algebroid
$A\Omega$, in the following way.  Put $A\Omega = \eval{V\alpha}{1(M)}
\subset T\Omega$, the bundle of tangent vectors to $\Omega$ at points
of the image of $1 : M \to \Omega$ which are vertical with respect to
the source projection $\alpha$; this is a vector bundle over $M$ in
the obvious way.  For any $g \in \Omega$ with $x = \alpha(g)$, right
translation by $g$ induces a map $r_{g*} : T_{1_x}\Omega \to T_g
\Omega$ which restricts to a map $V_{1_x}\alpha \to V_g \alpha$; in
this way, any section of $A\Omega \to M$ may be extended by right
translation to a vector field on $\Omega$ which is vertical over
$\alpha$.  The Lie bracket of any two such vector fields is again
vertical over $\alpha$, and so we may define the Lie bracket of two
sections to be the restriction, to $1(M)$, of the Lie bracket of the
corresponding right-invariant vector fields.  Finally, the map $a =
\eval{\beta_*}{A\Omega}$, the restriction of the tangent map to the
target projection $\beta$, satisfies the conditions for an anchor map.
We write $L\Omega$ for the kernel of the anchor map $a : A\Omega \to
TM$.

If the base manifold $M$ is connected then the Lie groupoid $\Omega$
is locally trivial if, and only if, the associated Lie algebroid
$A\Omega$ is transitive.

\section*{Appendix 2:\ some general results about vector bundles with connections}

\begin{prop}\label{parallel}
Let $E\to M$ be a vector bundle equipped with a linear connection. Denote by  
$\nabla$ the corresponding covariant derivative, and for any curve 
$c$ in $M$ let $\tau_c:E_x\to E_{y}$ (where $c(0)=x$, $c(1)=y$) be 
the isomorphism given by parallel translation along $c$. For any vector 
sub-bundle $E'$ of $E$, the following conditions are equivalent:
\begin{enumerate}
\item $\nabla(\Sec(E'))\subset\Sec(E')$; that is, for every section 
$\xi$ of $E'$ and every vector field $X$ on $M$, $\nabla_X\xi$ 
is again a section of $E'$;
\item for every curve $c$, $\tau_c$ restricts to an isomorphism of 
$E'_x$ with $E'_y$.
\end{enumerate}
\end{prop}

\begin{proof} 
$(1)\Rightarrow(2)$\quad By the usual properties of connections we can
define covariant derivatives of sections along curves.  Let
$\{\xi_a(t)\}$ be a local basis of sections of $E'$ along $c(t)$:\
then any local section of $E'$ along $c$ can be written as a linear
combination of the $\xi_a$ with coefficients which are functions of
$t$.  By assumption, $\nabla_{\dot{c}}\xi_a$ is a section of $E'$;
it may therefore be expressed as a linear combination of the
$\xi_a$, say $\nabla_{\dot{c}}\xi_a=K_{a}^b\xi_b$ for some
functions $K_a^b$ of $t$.  We shall modify the $\xi_a$ so as to
obtain a new local basis $\{\eta_a\}$ of sections of $E'$ along $c$ such that
$\nabla_{\dot{c}}\eta_a=0$.  Consider the equations
\[
\frac{d}{dt}(\Lambda^b_a)+\Lambda_a^cK_c^b=0
\]
for functions $\Lambda^b_a$ of $t$.  These are first-order linear
ordinary differential equations for the unknowns $\Lambda^b_a$, and
have a unique solution for initial conditions specified at $t=0$,
which we may take to be $\Lambda^b_a(0)=\delta^b_a$:\ then the
solution, considered as a matrix, will be nonsingular at all points of
$c$ sufficiently close to $x$.  Without loss of generality we may
assume that $y$ lies within the neighbourhood on which
$(\Lambda^b_a)$ is nonsingular.  Then if 
$\eta_a(t)=\Lambda_a^b(t)\xi_b(t)$,
$\{\eta_a\}$ is a new basis of sections of $E'$ along $c$ such that
\[
\nabla_{\dot{c}}\eta_a=\Lambda_a^b\nabla_{\dot{c}}\xi_b
+\frac{d}{dt}(\Lambda_a^b)\xi_b=0.
\]
But this is just the condition for $\eta_a$ to be parallel along $c$,
that is, for $\tau_{c}\eta_a(0)=\eta_a(1)$.  So there are bases of
$E'_x$ and $E'_{y}$ with respect to which $\tau_{c}$ is represented by
the identity matrix.

$(2)\Rightarrow(1)$\quad Let $c$ be smooth a curve in $M$ and $\eta$ a
local section of $E'$ over some neighbourhood of $c(0)=x$, so that
$\eta(t)=\eta_{c(t)}$ depends smoothly on $t$ and $\eta(t)\in
E'_{c(t)}$.  Now
\[
\nabla_{\dot{c}(0)}V=
\frac{d}{dt}\left(\left(\tau_{c}|_0^t\right)^{-1}\eta(t)\right)_{t=0},
\]
where $\tau_{c}|_0^t: E'_x\to E'_{c(t)}$ is the
isomorphism determined by parallel transport.  Thus
$\left(\tau_{c}|_0^t\right)^{-1}\eta(t)\in E'_x$ for all $t$,
and so $\nabla_{\dot{c}(0)}\eta\in E'_x$.
\end{proof}

\begin{prop}\label{vbundle}
Let $\pi:E\to M$ be a vector bundle equipped with a linear connection
$\nabla$.  Let $E'$ be a subset of $E$ such that
\begin{enumerate}
\item $\pi$ maps $E'$ onto $M$;
\item for each $x\in M$, $E'_x=\pi|_{E'}^{-1}(x)$ is a linear 
subspace of $E_x$;
\item for any pair of points $x,y\in M$ and any curve $c$ in $M$ 
joining them, $\tau_c(E'_x)\subseteq E'_y$. 
\end{enumerate}
Then $E'$ is a vector sub-bundle of $E$ and
$\nabla(\Sec(E'))\subset\Sec(E')$.
\end{prop}

\begin{proof}
For any curve $c$, $\tau_c:E_x\to E_y$ has an inverse, namely
$\tau_{\bar{c}}$, where $\bar{c}$ is the curve with the same image as
$c$ but followed in the opposite direction:\ $\bar{c}(t)=c(1-t)$.
Thus any point $e\in E'_y$ is the image under $\tau_c$ of a point of
$E'_x$, namely $\tau_{\bar{c}}(e)$; and so $\tau_c: E'_x\to E'_y$ is
both injective and surjective and so an isomorphism. 

We use the ray construction, as in the proof of
Proposition~\ref{isotriv}.  Fix a point $o\in M$ and take a coordinate
neighbourhood $U$ of $o$ with $o$ as origin such that the image of $U$
is the open unit ball in $\R^m$.  For each $x\in U$ let $r_{x}$ be the
ray joining $o$ to $x$ with respect to the coordinates.  We further
assume that $U$ is contained in some neighbourhood over which $E$ is
locally trivial; then in terms of the coordinates on $M$ and the
corresponding fibre coordinates on $E$, for any $e\in E_o$,
$\tau_{r_{x}}e$ is the solution of a system of ordinary differential
equations in which the coordinates of $x$ play the role of parameters.
The solutions of such equations depend smoothly on parameters, so
if we set $\xi(x)=\tau_{r_{x}}e$ then $\xi$ is a smooth local section
of $E$.  Now choose a basis $\{e_a\}$ for $E'_o$, and for each $x\in
U$ set $\xi_a(x)=\tau_{r_{x}}e_a$.  Then $\{\xi_a(x)\}$ is a basis of
$E'_x$, and $\{\xi_a\}$ is a basis of smooth local sections of $E'$.

This construction provides at the same time local fibre coordinates on
$E'$ and bases of local sections of it.  Let $U$ be one coordinate
neighbourhood in $M$, with origin $o$, for such a system of local
sections, and $V$, with origin $p$, another.  We choose a basis
$\{e_a\}$ for $E'_o$ and a basis $\{f_a\}$ of $E'_p$ to begin the
construction in each case.  We may assume that $\{f_a\}$ is the
parallel translate of $\{e_a\}$ along some curve $c$ joining $o$ and
$p$ (to assume otherwise merely involves interposing an additional
constant non-singular coefficient matrix in the expression for
transition functions shortly to be derived).  We denote by $\{\xi_a\}$
the basis of local sections of $E'$ determined by $o$, $U$ and
$\{e_a\}$, and $\{\eta_a\}$ the basis determined by $p$, $\{f_a\}$ and
$V$.  Then for any $x\in U\cap V$,
$\eta_a(x)=\tau_{V,x}\tau_c\tau_{U,x}^{-1}\xi_a(x)$, where we have
indexed the parallel transport operators along rays by the relevant
coordinate patch.  But $\tau_{V,x}\tau_c\tau_{U,x}^{-1}$ is an
isomorphism of $E'_x$, whose matrix with respect to the basis
$\{\xi_a(x)\}$ depends smoothly on $x\in U\cap V$.  This shows that
$E'$ is a smooth vector sub-bundle of $E$.  By
Proposition~\ref{parallel}, $\nabla(\Sec(E'))\subset\Sec(E')$.
\end{proof}

\begin{cor}\label{cap}
Let $E\to M$ be a vector bundle equipped with a linear connection, 
and let $E^1$ and $E^2$ be vector sub-bundles of $E$ such that 
$\nabla(\Sec(E^i))\subset \Sec(E^)$, $i=1,2$. Then $E^1\cap E^2$ 
is a vector bundle.
\end{cor}

\begin{proof}
By Proposition \ref{parallel}, parallel translation preserves both $E^1$
and $E^2$, so preserves their intersection.  The result now follows 
from Proposition~\ref{vbundle}.
\end{proof}

\begin{prop}\label{Lie}
If $E\to M$ is a Lie algebra bundle and $\nabla$ a connection on $E$ then 
$\nabla$ is a Lie  connection if and only if for every curve $c$ in 
$M$, $\tau_c:E_x\to E_{y}$ is an isomorphism of Lie algebras.  
\end{prop}

\begin{proof}
Suppose that $\tau_c$ is always a Lie algebra isomorphism.  Take a
basis $\{e_a\}$ for the Lie algebra $E_x$, and set $e_a(t)=\tau_c|_0^t
e_a$, where $\tau_c|_0^t$ is parallel translation along $c$ from
$x=c(0)$ to $c(t)$.  Then $\{e_a(t)\}$ is a basis for $E_{c(t)}$, and 
$[e_a(t),e_b(t)]=\tau_c|_0^t[e_a,e_b]$. So if we set 
$[e_a,e_b]=C^c_{ab}e_c$ then $[e_a(t),e_b(t)]=C^c_{ab}e_c(t)$. Now 
$\nabla_{\dot{c}}e_a(t)=0$, and so for any section $\xi(t)$ of $E$ 
along $c$, say $\xi(t)=\xi^a(t)e_a(t)$, we have 
$\nabla_{\dot{c}}\xi=\dot{\xi}^ae_a$. But for sections 
$\xi$, $\eta$ along $c$ we have 
$[\xi,\eta]=\xi^a\eta^bC^c_{ab}e_c$, whence
\[
\nabla_{\dot{c}}[\xi,\eta]=\frac{d}{dt}(\xi^a\eta^b)C^c_{ab}e_c
=(\dot{\xi^a}\eta^b+\xi^a\dot{\eta}^b)C^c_{ab}e_c
=[\nabla_{\dot{c}}\xi,\eta]+[\xi,\nabla_{\dot{c}}\eta].
\]
Conversely, suppose that $\nabla$ is Lie.  Let $\{e_a(t)\}$ be a basis
of parallel sections of $E$ along $c(t)$, so that 
$\nabla_{\dot{c}}e_a=0$. Then if we set
$[e_a(t),e_b(t)]=C_{ab}^c(t)e_c(t)$, it is evident from
the Lie property that the $C^c_{ab}$ are actually constants. But 
$e_a(t)=\tau_c|_0^te_a(0)$, so
\begin{align*}
[\tau_c|_0^te_a(0),\tau_c|_0^te_b(0)]&=[e_a(t),e_b(t)]\\
&=
C_{ab}^c(t)e_c(t)=C_{ab}^c(0)\tau_c|_0^te_c(0)=\tau_c|_0^t[e_a(0),e_b(0)],
\end{align*}
which says that $\tau_c|_0^t$ is a Lie algebra isomorphism.
\end{proof}

\begin{cor}\label{capLie}
Let $E\to M$ be a Lie algebra bundle equipped with a Lie connection, 
and let $E^1$ and $E^2$ be Lie algebra sub-bundles of $E$ such that 
$\nabla(\Sec(E^i))\subset \Sec(E^i)$, $i=1,2$. Then $E^1\cap E^2$ 
is a Lie algebra bundle.
\end{cor}

\begin{proof}
By Corollary \ref{cap}, $E^1\cap E^2$ is a vector bundle. Each fibre 
is a Lie algebra, and by Proposition~\ref{Lie} the fibres are 
pairwise isomorphic as Lie algebras.
\end{proof}

\subsubsection*{Acknowledgements} The first author is a Guest
Professor at Ghent University:\ he is grateful to the Department of
Mathematics for its hospitality.  The second author acknowledges the
support of grant no.\  201/09/0981 for Global Analysis and its
Applications from the Czech Science Foundation.

\subsubsection*{Address for correspondence}
M.\ Crampin, 65 Mount Pleasant, Aspley Guise, Beds MK17~8JX, UK\\
Crampin@btinternet.com

\end{document}